\documentclass[11pt,a4paper,reqno]{article} 

\usepackage[latin,english]{babel}
\usepackage[utf8]{inputenc} 

\usepackage[left=2.5cm,right=2.5cm,top=2.5cm,bottom=3.5cm]{geometry}

\usepackage{graphicx} 

\usepackage{booktabs} 
\usepackage{array} 
\usepackage{paralist} 
\usepackage{verbatim} 
\usepackage{subfig} 
\usepackage{amsmath}
\usepackage{amsfonts}
\usepackage{amssymb}
\usepackage{mathrsfs}
\usepackage{amsthm, math}
\usepackage{soul}

\usepackage{blindtext}
\usepackage{hyperref}

\usepackage{fancyhdr} 
\usepackage{sectsty}
\allsectionsfont{\sffamily\mdseries\upshape} 

\usepackage[nottoc,notlof,notlot]{tocbibind} 
\usepackage[titles,subfigure]{tocloft} 

\usepackage{xcolor}

\numberwithin{equation}{section}

\renewcommand{\cW}{\boldsymbol{\mathcal W}}

\newtheorem{theorem}{Theorem}[section]
\newtheorem{proposition}[theorem]{Proposition}
\newtheorem{lemma}[theorem]{Lemma}

\newtheorem{definition}[theorem]{Definition}
\theoremstyle{definition}
\newtheorem{remark}[theorem]{Remark}

\title{\bf Bifurcation of 
gravity-capillary Stokes waves \\
with constant vorticity}
\author{T. Barbieri, M. Berti, A. Maspero, M. Mazzucchelli}
\date{} 

\begin{document}
\maketitle

\noindent
{\bf Abstract.}
We  consider the gravity-capillary water waves equations of a 2D 
fluid with constant vorticity.  
By employing
variational methods we prove 
the bifurcation of periodic traveling water waves 
--which are steady in a moving frame--
for {\it all} the values of 
gravity,  surface tension, constant vorticity, 
depth  and wavelenght, 
extending   
previous results valid for restricted values of the  parameters.  We parametrize the bifurcating 
Stokes waves either 
with their 
speed or  their  momentum.


\\[1mm]

\section{Introduction and main results}

A very classical   fluid mechanics problem 
regards 
the search for traveling 
 surface waves. In this paper 
we consider the Euler equations  for a 2-dimensional 
incompressible and inviscid fluid with constant vorticity $\gamma $, under the action of 
gravity $ g > 0 $ and surface tension $ \kappa \geq 0  $ 
at the free surface.
The fluid 
occupies the region
\begin{equation}
\label{domain}
\cD_{\eta, \tth} := \big\{ (x,y)\in (\lambda \T) \times \R \ : \ -\tth <  y<\eta(t,x) \big\} \, , 
\quad \T := \T_x :=\R/ (2\pi 
 \Z) \, ,
\end{equation} 
with  a, possibly infinite, 
depth $\tth > 0 $  and 
 space periodic boundary conditions with wavelength $ 2 \pi \lambda >  0 $.
 
The goal  is to show that variational methods, based on the Hamiltonian formulation of the 
water waves equations
\cite{Z,CS,CIP,Wh},  
allow to prove 
the bifurcation of periodic traveling water waves -called Stokes waves- 
for {\it all} the values of 
gravity $ g > 0$,  surface tension $ \kappa \geq 0 $, constant vorticity $ \gamma \in \R $, 
depth $ \tth\in (0,+\infty] $ and wavelenght $ 2 \pi \lambda > 0 $
(clearly not all the physical parameters   $ g, \tth, \kappa, \gamma, \lambda $
   are  independent), 
   see Theorem \ref{simpl0}. These solutions look steady in a reference frame moving with the speed of the wave.
Previous results as \cite{W,M},  based on the use of the 
Crandall-Rabinowitz bifurcation theorem from the simple eigenvalue, restrict the range of allowed parameters. 
We shall prove
the existence of 
non-trivial Stokes waves 
parametrized by the speed, see  
Theorem \ref{simpl},  or the momentum, see Theorem \ref{simplCN}. 


\smallskip

The literature concerning  Stokes waves is enormous. We 
refer to \cite{const_book,HHSTWWW} for extended presentations. Here we only mention that,
after the pioneering work of Stokes  \cite{stokes}, the first rigorous construction 
of small amplitude space periodic traveling waves is due to  
Nekrasov \cite{Nek}, Levi-Civita \cite{LC} and Struik \cite{Struik} for  irrotational $2$D flows  under the action of  gravity.  
Later  Zeidler \cite{Zei} considered the effect of capillarity,  see also Jones-Toland \cite{JT},  
and Dubreil-Jacotin \cite{dubreil}, Goyon \cite{goyon} of  vorticity.
 More recent results  for capillary-gravity waves with  vorticity 
 have been given in Wahl\'en \cite{Wh0,W} and  
  Martin  \cite{M}. 
All these works deal with 2D water waves, and  can ultimately 
be  deduced by the Crandall-Rabinowitz 
bifurcation theorem from a simple eigenvalue. 

These local bifurcation results have been extended to  global branches 
of steady waves, as started in the celebrated works of 
Keady-Norbury \cite{KN}, Toland \cite{To}, Amick-Fraenkel-Toland \cite{AFT},  McLeod \cite{ML},
Plotinkov \cite{P1982}
for irrotational flows and  Constantin-Strauss \cite{CSt},
Constantin-Strauss-Varvaruca \cite{CSV}, 
for fluids with vorticity. See also the recent works by Wahl\'en-Weber \cite{WW24} and Kozlov-Lokharu \cite{KK} for general vorticity. 

For  three dimensional irrotational fluids, bifurcation of small amplitude
bi-periodic traveling 
waves   has been proved  in Reeder-Shinbrot \cite{RS}, Craig-Nicholls \cite{CN,CN2} 
for gravity-capillary  waves by  variational methods
and  by Iooss-Plotnikov \cite{IP-Mem-2009,IP2} for gravity waves (this is a small divisor problem). 
We also quote 
the  results \cite{LSW, GNPW} for  doubly periodic gravity-capillary Beltrami flows.

We finally mention that in the last years also 
the existence of quasi-periodic  traveling Stokes waves --which are the nonlinear superposition of 
 Stokes waves moving 
with rationally independent speeds--
 has been proved in \cite{BFM,FG,BFM2}  by means of KAM methods. 
 
\smallskip

The results of the present paper,
Theorems \ref{simpl}, \ref{simplCN}, are not covered 
by 
Craig-Nicholls \cite{CN} as we include constant vorticity effects for $ 2D$ fluids 
and we parametrize the Stokes waves 
also with their speed. This requires a different critical point theory.
With respect to Martin \cite{M} we cover all the possible  
bifurcation speeds, also the resonant ones, that we characterize in Proposition \ref{KER24}.
Let us now present rigorously the results and the techniques.
\\[1mm]
{\bf The water waves equations.}
In the sequel, 
with no loss of generality,  we  set $ \lambda = 1 $. 
The unknowns of the problem are the free surface  $ y = \eta (t, x)$
of the time dependent domain $\cD_{\eta,\tth} $ in \eqref{domain} and the 
divergence free  velocity field $\vect{u(t,x,y)}{v(t,x,y)}$.
If the  fluid has constant vorticity 
$ v_x - u_y = \gamma $, 
the velocity field   is the sum of the Couette flow $\vect{-\gamma y}{0}$,  which carries
all the  vorticity $ \gamma $ of the fluid,  and an irrotational field $ \nabla_{x,y} \Phi (t, x,y) $. 
Given $ \psi (t,x) := \Phi (t,x, \eta(t,x)) $ 
one recovers $ \Phi $ by solving the elliptic problem
$$
\Delta \Phi = 0  \ \mbox{ in } \cD_{\eta, \tth} \, , \quad
\Phi = \psi \  \mbox{ at } y = \eta(t,x) \, , \quad
\Phi_y \to  0  \  \mbox{ as } y \to  - \tth \, .
$$
Imposing  
that the fluid particles at the free surface remain on it along the evolution
(kinematic boundary condition), and that
the pressure of the fluid 
 is  constant  at the free surface (dynamic boundary condition),  the 
time evolution of the fluid is determined by the 
system of equations 
\begin{equation}
\label{WWZakharov}
\begin{cases}
\eta_t = G(\eta)\psi + \gamma \eta \eta_x \\
\displaystyle{\psi_t = - g\eta  - \frac{\psi_x^2}{2} + 
	\frac{(  \eta_x \psi_x + G(\eta)\psi)^2}{2(1+\eta_x^2)}  + \kappa\partial_x \Big(\frac{\eta_x}{\sqrt{1+\eta_x^2}}\Big)+\gamma \eta \psi_x  + \gamma \partial_x^{-1} G(\eta) \psi} \, ,
\end{cases}
\end{equation}
where $G(\eta)$ is the  Dirichlet-Neumann operator  
$$
G(\eta)\psi := G(\eta,\tth)\psi :=
 (- \Phi_x \eta_x + \Phi_y)\vert_{y = \eta(x)} \, . 
$$
In \eqref{WWZakharov} 
 $ (\pa_x^{-1} f) (x) $ 
 denotes the unique 
primitive with  zero average  of a $ 2 \pi  $-periodic zero average function $ f(x) $.  
It turns out that $ G(\eta)\psi  $ has zero average. As consequence 
the average of $ \eta (x) $,   
$ \langle \eta \rangle := \eta_0 := 
\frac{1}{2  \pi } \int_\mathbb{ T} \eta (x) \, d x $ 
  is a prime integral of
  \eqref{WWZakharov}. 
  Note also that, since  
  $ G(\eta) [1] = 0 $ vanishes on the constants,  
the vector field in the right hand side of \eqref{WWZakharov} does not depend on $ \psi_0 = \la \psi \ra $.
 \\[1mm]
{\bf Hamiltonian structure.}
As observed in the irrotational case  by Zakharov \cite{Z}, Craig-Sulem \cite{CS}, and  in presence of constant vorticity by  Constantin-Ivanov-Prodanov \cite{CIP} and 
Wahl\'en \cite{Wh}, 
the water waves equations \eqref{WWZakharov} are the Hamiltonian system 
with a non-canonical structure
\begin{equation}\label{WWZakharovHam}
	\eta_t = \nabla_{\psi} H(\eta,\psi)\,, \quad \psi_t = (-\nabla_{\eta} + \gamma \pa_x^{-1}\nabla_{\psi})H(\eta,\psi)\,,
\end{equation}
where $\nabla$ denotes the $L^2$-gradient, with 
Hamiltonian
\begin{equation}\label{Hamiltonian}
	H(\eta,\psi) 
            = \int_{ \T} 
            \frac12 \Big( \psi \,G(\eta) \psi + g \eta^2 \Big) 
	 +\kappa \big(\sqrt{1+\eta_x^2}-1\big)
  + \frac{\gamma}{2} 
	\big( -   \psi_x \eta^2 + \frac{\gamma}{3} \eta^3 \big) 
d x\,.
\end{equation} 
The equations \eqref{WWZakharov} simplify considerably introducing as in  Wahl\'en \cite{W}  the  variable 
\begin{equation}\label{walhen}
	\zeta :=
	\psi-\frac{\gamma}{2}\partial_x^{-1} \big(  \eta - \langle \eta \rangle \big)   = \psi-\frac{\gamma}{2}\partial_x^{-1}\Pi_0^\bot \eta 
\end{equation}
	where $\Pi_0^\bot $ is the  projector on the zero average functions,  
 $ \Pi_0^\bot := \uno  - \Pi_0 $ and $ \Pi_0 \eta :=  \la \eta \ra $.
Actually,  under the linear change of variable  
\begin{equation}  \label{wahlen}
 \vect{\eta}{\psi} =  W \vect{\eta}{\zeta} \, , \quad  
		W=\left(
			\begin{matrix}{}
				I     					   &   0\\
				\frac{\gamma}{2}\partial_x^{-1}\Pi_0^\bot  &   I
			\end{matrix}
		\right)\, ,\quad 
        W^{-1}=\left(
			\begin{matrix}{}
				I     					   &   0 \\
				-\frac{\gamma}{2}\partial_x^{-1}\Pi_0^\bot  &   I
			\end{matrix}
		\right) \, , 
	\end{equation}
 the Hamiltonian system \eqref{WWZakharovHam} assumes the canonical Darboux form  
	\begin{equation}\label{EqWWZakharovHamNewCoordinates}
		\partial_t\left(
			\begin{matrix}{}
				\eta     \\
				\zeta
			\end{matrix}
		\right)= J \left(
			\begin{matrix}{}
				\nabla_\eta \mathcal{H}(\eta,\zeta)     \\
				\nabla_\zeta \mathcal{H}(\eta,\zeta)
			\end{matrix}\right)
   \qquad \text{where} 
   \qquad 
   J :=\left(\begin{matrix}
                    0   &   I   \\
                    -I  &   0
                    \end{matrix}\right)            \end{equation}
 is the canonical Poisson tensor and the new Hamiltonian is 
$\mathcal{H}(\eta,\zeta):= H\circ W(\eta,\zeta)$ .
  
The symplectic form associated to the Hamiltonian structure of \eqref{EqWWZakharovHamNewCoordinates} is 
    \begin{equation}\label{W}
		\cW\left(\vect{\eta}{\zeta},\vect{\eta_1}{\zeta_1}\right)=
  \left\langle J^{-1}
  \vect{\eta}{\zeta}, 
  \vect{\eta_1}{\zeta_1}\right\rangle
  =
        \frac{1}{2\pi}    \int_{ \T} \eta(x) \zeta_1(x)-\eta_1(x) \zeta(x)\, d x 
    \end{equation}
    where 
    $
    \langle f  , g \rangle :=
    \frac{1}{2\pi}    \int_{ \T} f (x) \cdot g (x) d x $  
    denotes the $L^2(\mathbb{T}, \mathbb{R}^2)$ real scalar product, so that
    \begin{equation}\label{sympl gradient}
    \di_u \cH(u)[ \cdot ] = 
    \cW ( J \nabla \cH (u), \cdot ) \, , 
    \end{equation}
   i.e. the Hamiltonian vector 
    $J \nabla \cH (u) $ field is the symplectic gradient of $ \cH (u) $. 
 Note that the vector field $ J \grad \cH (\eta, \zeta)$ in \eqref{EqWWZakharovHamNewCoordinates} does not depend on $ \zeta_0 = \la \zeta \ra $ and the first component has zero average, so that  the average 
$ \la  \eta \ra  $ 
is a prime integral of
\eqref{EqWWZakharovHamNewCoordinates}, as well as for \eqref{WWZakharov}. 
 \\[1mm]
{\bf Reversible structure and $ O(2)$-symmetry.} 
The Hamiltonian system \eqref{EqWWZakharovHamNewCoordinates} possesses 
both a $ \Z_2 $ and an  $ {\mathbb S}^1 $-symmetry. 
Indeed it is  reversible, namely 
	\begin{equation}\label{sym}
\cH \circ 	\mathscr{S} = \cH 
\qquad \text{where} \qquad 
  \mathscr{S}
		\left(\begin{matrix}{}
				\eta\\
				\zeta
			\end{matrix}
		\right)(x) :=
		\left(\begin{matrix}{}
				\eta(-x)\\
				-\zeta(-x)
			\end{matrix}
		\right) \, , 
	\end{equation}
    (note that $ \mathscr{S} = \mathscr{S}^{-1}$), 
 and, 
since the bottom of the fluid domain is flat, it is space invariant, namely 
\begin{equation}\label{sym tau}
\cH \circ \tau_\theta = \cH 
\qquad \text{where} \qquad 
		\tau_\theta
		\left(\begin{matrix}{}
				\eta\\
				\zeta
			\end{matrix}
		\right)(x) :=
		\left(\begin{matrix}{}
				\eta(x-\theta)\\
				\zeta(x-\theta)
			\end{matrix}
		\right) \, , \quad  \forall \theta\in \mathbb{R} \, .  
	\end{equation} 
On a phase space of  $ 2 \pi $-periodic  functions
$ (\eta (x), \zeta(x) )$,  
the translation operators $ \tau_\theta $ in \eqref{sym tau}
  form a representation of 
  the group $ \mathbb{S}^1 := \T = \R \slash (2 \pi \Z) $ 
  into the group of linear transformations of the phase space, and we write $ (\tau_\theta)_{\theta \in \mathbb{S}^1} $ .
Note also that $ (\tau_\theta)_{\theta \in \mathbb{S}^1} $ 
is a one parameter group of symplectic maps. 
According to Noether theorem the associated first integral 
of
\eqref{EqWWZakharovHamNewCoordinates} is the momentum 
	\begin{equation}\label{DefI}
		\mathcal{I}(\eta,\zeta):=\int_{ \mathbb{T}}\eta_x (x)  \zeta(x)  \, d x \, , 
	\end{equation}
since the  Hamiltonian vector field
generated by $\cI$ is  
\begin{equation}\label{Jetazeta}
 J\nabla \mathcal{I}(\eta,\zeta)=\partial_x\vect{\eta}{\zeta} \, , 
\end{equation}
 which is the generator of the group of the  translations  \eqref{sym tau}.
Also the 
momentum $ \cI $ clearly satisfies 
	\begin{equation}\label{IsItau}
	\mathcal{I}\circ\mathscr{S}=\mathcal{I} \, , 
 \quad 
 \mathcal{I}\circ\tau_\theta=\mathcal{I}\:, \quad \forall \theta \in 
 \mathbb{S}^1\, . 
	\end{equation}
These joint symmetries actually amount to the fact that $ {\cal H} $ and $  \mathcal{I} $ 
are invariant under the 
action  of the orthogonal group $ O(2) \cong \mathbb{S}^1 \rtimes \Z_2 $, cfr. Remark \ref{O2action}.
\\[1mm]
{\bf Traveling waves.}	
We seek for periodic traveling waves
 of \eqref{EqWWZakharovHamNewCoordinates}, i.e. solutions of the form $\eta(x-ct)$ and 
	$\zeta(x-ct)$ where
 $ \eta (x) , \zeta (x) $ are $ 2 \pi $-periodic.
	Substituting inside \eqref{EqWWZakharovHamNewCoordinates} we obtain
\begin{equation*}
		-c\:\partial_x\left(
			\begin{matrix}{}
				\eta	\\
				\zeta
			\end{matrix}
		\right)= J \left(
			\begin{matrix}{}
				\grad_\eta \cH(\eta,\zeta)     \\
				\grad_\zeta \cH(\eta,\zeta)
			\end{matrix}\right) \, , 
\end{equation*}
which, in view of 
\eqref{WWZakharov}, \eqref{wahlen},  \eqref{Jetazeta}
is 
	equivalent to find  
 solutions 
$ c \in \R $ and 
 $u=(\eta,\zeta) $  of the nonlinear equation 
 \begin{equation}\label{eq:H}
 \qquad 
  \cF(c,u):=  c \pa_x u + 
  J  \grad \cH (u)  =  
  J(\grad \cH + c \grad \cI)(u) 
  = 0  
  \, . 
 \end{equation}
We regard $ \cF (c, \cdot)  $ as a map defined in  
a dense subset $ X \subset L^2( \T) \times L^2_0(\T)$  that we define below in \eqref{spaceX},
with values in the  space $ Y $ that we will introduce in \eqref{spaceY}
, \begin{equation}\label{spaziY}
  \cF\colon \R \times X \to 
 Y \subset L^2_0( \T) \times L^2( \T) \, , \quad
 (c, u) \mapsto  \cF (c, u ) \, ,
\end{equation}
where $L^2_0(\T)  $ is the subspace of $ L^2 (\T)$ of  zero average functions.
Since $ J$ is invertible, the equation \eqref{eq:H} is equivalent 
 to search 
 critical  points,
 i.e. equilibria, of 
 the Hamiltonian 
	\begin{equation}\label{Variational}
\Psi (c, \cdot ) : X \to \R \, , \quad 	
 \Psi (c, u ) := 
 (\cH+c \cI)(u) \, , 
	\end{equation}
for some value of the moving frame 
speed $ c $. 

By the group symmetries \eqref{sym tau}, \eqref{sym}, \eqref{IsItau},  
if $ u $ is a Stokes wave solution of
\eqref{eq:H} then each  translated function 
$ \tau_\theta u $ and  reflected one 
$ \mathscr{S} u $ are solutions as well.  
We shall say that two non-trivial Stokes waves   
solutions of \eqref{eq:H}  
are {\it geometrically  distinct}  if they are not obtained 
by applying the translation operator 
$ \tau_\theta $ or the reflection operator $ \mathscr{S} $ to the other one. 
\\[1mm]
{\bf Functional setting.} 
We define 
for $  \sigma > 0 $ and $ s \in \R $, the Hilbert space  
\begin{equation}\label{spaceX}
X := H^{\sigma,s} \times H^{\sigma,s}_0  := 
H^{\sigma,s}(\T) \times H^{\sigma,s}_0(\T)   
\end{equation}
where  $H^{\sigma,s}(\T)$ is the space of $ 2 \pi $-periodic analytic functions 
$ u(x) = \sum_{k \in \mathbb{Z}} u_k e^{\im k x} $ 
with norm
$$ \| u \|_{\sigma,s}^2 := \sum_{k \in \mathbb{Z}} |u_k|^2 \langle k \rangle^{2s} 
e^{2 \sigma |k|} < + \infty \,  , \quad \la k \ra := \max\{1,|k|\} \, , 
$$
and $H^{\sigma, s}_0:= H^{\sigma,s} \cap L^2_0(\T)$.
For any $ \sigma \geq 0 $ and
$ s > 1/ 2 $ each space 
$H^{\sigma,s}(\T)$ is an algebra with respect to the product of functions.

Throughout  the  paper we assume that the space $ X $ in  \eqref{spaceX} has 
 Sobolev exponent  
$ s \geq 7 / 2 $ with $ s + \tfrac12 \in \N $. 
This allows us to directly apply  
\cite{BMV2}[Theorem 1.2] 
about the analyticity  of the Dirichlet-Neumann operator in the proof of Lemma \ref{range equation}. 
We remark that this is actually  not restrictive since $ \sigma > 0 $ is arbitrary and, for any   
$ \sigma' < \sigma $, we have  
$ H^{\sigma',s'} (\T) \subset H^{\sigma,s} (\T) $ for any $  s' \in \R $.

The target space in \eqref{spaziY} is, in view of \eqref{WWZakharov} and \eqref{walhen}, 
\begin{equation}\label{spaceY}
Y := H^{\sigma,s-1}_0(\T) \times H^{\sigma,s-2}(\T) \mbox{ if } \kappa >0 \, , \ \ 
Y := H^{\sigma,s-1}_0(\T) \times H^{\sigma,s-1}(\T) \mbox{ if } \kappa = 0 \, . 
\end{equation}
Note that  $ \cF (c,0) = 0 $ for any $ c \in \R $.  
We are going to prove, 
by means of variational arguments,
that, for {\it any} value of the parameters 
$$ 
(g,\tth, \kappa, \gamma) \in 
(0,+\infty) 
\times (0,+\infty] 
\times [0,+\infty) 
\times \R \, , 
$$
{\it any} point $(c_*, 0) $ where $  {\cal L}_{c_*} := \di_u \cF(c_*,0)$ is not invertible
(this is a necessary condition for bifurcation)
is actually a point of bifurcation of non-trivial Stokes waves solutions of  
$ \cF(c,u) = 0 $. 
We now present  in detail the main results and techniques of proof. 
\\[1mm]
{\bf Main results.}
In Section \ref{sec:linea} we  diagonalize the 
operator $  \di_u \cF (c,0) $,  written explicitly in \eqref{dFc0}-\eqref{illinea}, and we 
prove that all the possible speeds  of bifurcation form  the set 
\begin{equation}\label{bifspeed}
\mathsf C := \left\{ 
 \frac{\Omega_j(g, \tth, \kappa, \gamma)}{j} \colon j \in \Z\setminus \{0\}\right\} 
\end{equation}
where  
 $\Omega_\xi  :=
  \Omega_\xi (g, \tth, \kappa, \gamma) $, 
  $ \xi \neq 0 $,  is the dispersion relation 
  of the gravity-capillary water waves equations with constant vorticity  
    \begin{equation}\label{Omega}
		\Omega_\xi  
  :=
  \begin{cases}
			\frac{\gamma}{2}\tanh(\tth \xi )+ 
					\sqrt{\big(g+ \kappa \xi^2)
     \xi \tanh(\tth \xi ) +
     \frac{\gamma^2}{4} 
     \tanh^2 (\tth \xi)} \, ,
     \quad \text{if} \   
     \tth < + \infty \\
     \frac{\gamma}{2}
\sign(\xi ) + 
					\sqrt{\big(g+ \kappa \xi^2) |\xi| 
     +\frac{\gamma^2}{4} } \, ,
     \qquad \qquad \qquad
     \qquad \quad \    \   \text{if} \ 
     \tth = + \infty \, .
     \end{cases}
	\end{equation}
Sometimes, in the Stokes waves literature, 
the function $j \mapsto \frac{\Omega_j}{j} $ is called the dispersion relation as well.
    
Then we fix any ``wave vector" $j_* \in \Z\setminus \{0\}$ and 
we consider the bifurcation speed 
\begin{equation}\label{def:c*}
c_*= \frac{\Omega_{j_*}}{j_*} \in \mathsf C \, ,
\end{equation}
for which $ {\cal L}_{c_*}$ possesses a non-trivial kernel, see \eqref{cLincl}.
In this paper we prove the following result, further detailed in Theorems \ref{simpl} and \ref{simplCN}.

\begin{theorem}\label{simpl0} 
For any value of gravity, depth, surface tension, vorticity
$ (g,\tth, \kappa, \gamma) \in 
(0,+\infty) 
\times (0,+\infty] 
\times [0,+\infty) 
\times \R $ 
and any wave vector $ j_* \in \Z \setminus \{ 0 \} $ there exist small amplitude non-trivial 
analytic Stokes waves solutions $ u \in X $ 
of $ \cF(c,u)= 0 $ with speed $c $ close  to
$ c_* := c_*(g,\tth, \kappa, \gamma,j_*)$.
\end{theorem}

The kernel of $ \cL_{c_*}$ has 
real dimension at least $2 $.
If  ker$(\cL_{c_*}) = 2 $ (non-resonant case) the
existence of Stokes waves was proved  in \cite{W,M}.
The main novelty of the paper  is to 
prove bifurcation 
of Stokes waves
also 
when 
dim ker$(\cL_{c_*}) > 2 $, called  ``resonant bifurcations".
In order to see if 
ker$({\cal L}_{c_*})$ is  2 dimensional 
or has a higher dimension,  we have to determine if there exist   
other  integers $ j $ such that  
\begin{equation}
\label{omegainter}
    \frac{\Omega_j}{j} = c_* = 
\frac{\Omega_{j_*}}{j_*} \, . 
\end{equation}
In  Proposition \ref{KER24} we fully characterize
the values of  $ g , \tth , \kappa , \gamma $ such that  
dim ker$({\cal L}_{c_*}) = 2 $  
or it is higher dimensional, in such a case  it turns out that 
dim ker$({\cal L}_{c_*}) = 4 $ (resonant case). 
We introduce, for any depth $ \tth < + \infty $,  the ``vorticity modified-Bond" numbers
\begin{equation}\label{def:bond}
     \tB_\pm(g,\tth, \kappa, \gamma):=  \frac{\kappa}{g\tth^2}-\frac{\tth\gamma^2}{6g}\Big(1\pm \sqrt{1 + \frac{4g}{\tth \gamma^2} } \ \Big) \, ,
\end{equation}
which reduce, in the irrotational case $\gamma =0$, 
to the ``classical'' Bond number
$\tB_\pm(g,\tth, \kappa, 0) = \frac{\kappa}{g\tth^2} $.
In Section \ref{sec:dispe} we prove the following result. 

 \begin{proposition}\label{KER24}
 {\bf (Kernel of $\cL_{c_*} $).}
For any  $ g > 0 $, $ \kappa \geq 0 $,  
$ \tth \in (0,+\infty] $,
$ \gamma \in \R $, $j_* \in \Z\setminus\{0\}$, 
let $c_* := c_*(g,\tth, \kappa, \gamma,j_*)$ be the speed defined in \eqref{def:c*}. The following holds true: if 
\begin{itemize}
    \item[(i)] either $\tth < \infty$, $\kappa >0$ and 
the vorticity-Bond numbers 
$\tB_\pm(g,\tth, \kappa, \gamma)$  in \eqref{def:bond}
fulfill
 $$    \tB_+(g,\tth, \kappa, \gamma) < \tfrac13 \, \mbox { if } 
 \, c_* >0 \, , 
 \quad
  \tB_-(g,\tth, \kappa, \gamma) < \tfrac13 \, \mbox { if }
  \, c_* <0 \, ; 
 $$
 \item[(ii)] or $\tth = + \infty$  and $\kappa >0$;  
\end{itemize}
then the kernel of the operator $\cL_{c_*} $ has dimension  either two or four.
 In all other cases  dim ker$(\cL_{c_*}) = 2 $.  
For any pair of 
integers $ 1 \leq |j_* | < |j |$ 
with $\sign(j) = \sign (j_*) $
, for any  $ g > 0 $, $ \tth \in (0,+\infty] $, 
$ \gamma \in \R $, 
there exists $ \kappa >0  $ such that \eqref{omegainter} holds, in particular 
$ \dim \ker(\cL_{c_*}) = 4 $.  
    \end{proposition}

The non-resonant bifurcations when dim ker$(\cL_{c_*}) = 2 $
have been studied in \cite{W,M} by means of the Crandall-Rabinowitz
   bifurcation theorem, as we report in paragraph \ref{parNR}. 
   In this case a minor improvement 
   is to  prove also 
   the analyticity of the Stokes waves.  
The main novelty of this paper is to prove 
the bifurcation of 
Stokes waves in the resonant case when the kernel  of $ \di_u \cF(c_*,0) $    
is  $4$ dimensional. 
The key idea is to use 
critical point theory,
see  Section \ref{sec:Reso}.  
 
\smallskip
Our  first result 
proves the existence  of 
solutions of 
$ \cF (c, u) = 0 $ 
for any fixed speed 
$ c $ close to $ c_* $.
According to all the possible values  of 
$ g, \tth, \kappa,  \gamma  $ and $ j_* $,
the non-trivial Stokes wave 
solutions of \eqref{WWZakharov}  have speed either 
$ c = c_* $ (case ($i$)), or 
$ c < c_* $ (sub-critical bifurcation) or $ c> c_* $
(super-critical bifurcation) as described in  
cases ($ii$)-($iii$) below.

\begin{theorem}\label{simpl} 
{\bf (Stokes waves with fixed speed)}
For any 
$ (g,\tth, \kappa, \gamma) \in 
(0,+\infty) 
\times (0,+\infty] 
\times [0,+\infty) 
\times \R $ 
and any integer $ j_* \in \Z \setminus \{ 0 \} $ 
such that
$ \ker (\cL_{c_*})$ has dimension $4$
, the following
alternatives may occur.   
Either
\begin{description}
\item ($i$) 
$ u = 0 $ is a non-isolated 
Stokes wave solution
with speed $ c_* := c_*(g,\tth, \kappa, \gamma,j_*)$  defined in \eqref{omegainter}
of $ \cF (c_*, u) = 0 $;
\end{description}
or
\begin{description}
\item ($ii$) there is a one sided neighborhood ${\cal U}$ of $ c_* $ 
such that 
for any $  c \in  {\cal U}\setminus \{ c_* \} $ the equation $ \cF (c, u) = 0 $  
possesses  at least two geometrically distinct non-trivial 
analytic Stokes wave solutions $ u_1(c), u_2(c)  \in X $ with speed $ c $
, tending to $0 $  as $ c \to c_* $; 
\end{description}
or
\begin{description}
\item ($iii$) there is a neighborhood $ {\cal U}$ of $ c_* $ 
such that 
for any $ c \in  {\cal U}\setminus \{ c_* \} $ the equation 
$ \cF (c, u) = 0 $  possesses at least one non-trivial 
analytic Stokes wave  solution  $ u(c) \in X $ with speed $ c $
,
tending to $0 $  as $ c \to c_* $.
\end{description}
\end{theorem}

The proof, given in Section \ref{sec:speed},
is based on searching for mountain pass critical points of the
reduced Hamiltonian obtained after a variational Lyapunov-Schmidt reduction  
\`a la Fadell-Rabinowitz
\cite{FR,R3}. The most delicate case is ($iii$) 
for which we provide complete and self-contained proofs in Appendix \ref{sec:App}.
We hope that these powerful techniques could be more effectively 
used for bifurcation problems for fluid PDEs.   

In Section \ref{sec:mom} we look for solutions parametrized by the momentum. 

\begin{theorem}\label{simplCN} 
{\bf (Stokes waves with fixed momentum)}
For any 
$ (g,\tth, \kappa, \gamma) \in 
(0,+\infty) 
\times (0,+\infty]
\times [0,+\infty) 
\times \R $ 
and any integer $ j_* \in \Z \setminus \{ 0 \} $ 
such that
$ \ker ( \cL_{c_*}) $ has dimension $4$, there exists
$ a_0 > 0$
such that, 
for any $ a \in (-a_0, a_0) $, with 
$ \sign (a) = - \sign (j_*) $, 
there exist at least two geometrically distinct non-trivial 
analytic Stokes wave solutions $ u_1 (a), u_2 (a) \in  X  $
,
with 
momentum $ \cI ( u_i (a)  ) = a $ and speeds 
$ c_i (a) $, $ i = 1,2 $, namely  solutions 
$$ 
(c_i (a), u_i (a)) \,, 
\quad i = 1, 2 \, , \quad 
\text{of the 
equation} \quad \cF (c, u) = 0 \, , \quad 
\text{with} \quad 
\cI ( u_i (a)  ) = a \, .
$$
Furthermore
$ (c_i (a), u_i (a)) 
\stackrel{\R \times X}\longrightarrow  (c_*, 0)   $ as $ a \to 0 $
where $ c_* := c_*(g,\tth, \kappa, \gamma,j_*)$ is defined in \eqref{omegainter}.
\end{theorem}

In this case the proof is reduced to search for critical points 
of a reduced Hamiltonian 
on a sphere-like manifold 
(of functions with fixed momentum $ a $, cfr.  \eqref{sfereSa})
, in the spirit   
of  Weinstein-Moser resonant center theorems \cite{Mo,We1,We2}, see also \cite{Berti}, and    
Craig-Nicholls \cite{CN}. 

\begin{remark}
{\bf (Dependence of the Stokes waves 
with respect to the 
speed/momentum)}  Critical points 
of a functional may vary highly erratically with respect to arbitrarily small
perturbations of the functional. This is why, for  the resonant bifurcations
considered in Theorems \ref{simpl} and 
 \ref{simplCN}, 
the dependence of the Stokes waves $ u_i (c)$,  with respect to their speed $c $,  and of the Stokes waves $ u_i (a)$  (as well as their speed $ c_i (a) $)
with respect to their momentum   $ a $, may be rather irregular.
\end{remark}

The paper is organized as follows. 
In Section \ref{sec:linea} we characterize all the bifurcation speeds where $ \di_u \cF(c,0) $ is not invertible, i.e the set \eqref{bifspeed}.
In Section \ref{sec:dispe} 
we prove Proposition \ref{KER24} which establishes,
according to the values of $ g, \kappa, \gamma, \tth $, if the 
kernel of $ \di_u \cF(c_*,0) $ is 2 dimensional (non-resonant case)
or 4 dimensional (resonant case). 
In Section \ref{sec:LS} we perform the variational Lyapunov-Schmidt reduction.
In section \ref{sec:Reso}
we prove the bifurcation of 
Stokes waves in the resonant case 
either 
with fixed 
speed (Theorem \ref{simpl}) or with fixed momentum
(Theorem \ref{simplCN}). 
In Appendix \ref{sec:App} we prove the existence of Palais-Smale sequences
at the mountain pass level. 

\section{The linearized operator}\label{sec:linea}
We study the linearized operator
\begin{equation}\label{dFc0}
    \di_u \cF(c,0):= \cL_c := c \pa_x  + J \di_u \grad  \cH(0) 
\end{equation}
where, in view of  \eqref{Hamiltonian}, \eqref{wahlen},
\begin{equation}\label{illinea}
	 \di_u \grad \mathcal{H}(0)
  := W^\top \, \di_u \grad H(0) W 
  =\left(
			\begin{matrix}{}
				g -\kappa \pa_x^2-\frac{\gamma^2}{4}\Pi_0^\bot\partial_x^{-1}G(0)\partial_x^{-1}\Pi_0^\bot &   
                    -\frac{\gamma}{2}\Pi_0^\bot \partial_x^{-1}G(0)\\
				\frac{\gamma}{2}G(0)\partial_x^{-1}\Pi_0^\bot 		&  			G(0)
			\end{matrix}
		\right)
\end{equation}
and
$$ 
G(0) =  D \tanh (\tth D) \ \
\text{if} \    \tth < + \infty \, , \qquad 
G(0) = |D|  \ \   
\text{if} \ \tth = + \infty \, , 
$$ 
is the Dirichlet-Neumann operator at the flat surface. 
Note that  the real operator $ \di_u \grad \mathcal{H}(0)$ is 
 symmetric  and  that 
	\begin{equation}\label{dNablaI}
		c \pa_x = J \di_u\nabla \mathcal{I}(0) \ , 
  \quad 
\di_u\nabla \mathcal{I}(0)  =\left(
			\begin{matrix}{}
				0 &  	 -\partial_x\\
				\partial_x & 	 0
			\end{matrix}
		\right)  \ , 
	\end{equation}
 so the linear operator $\cL_c$ in \eqref{dFc0} is Hamiltonian, i.e. of the form $ J A $ where 
 $ A $ is a symmetric  operator.
 
We diagonalize $\cL_c$ only on the zero average functions. The action of $\cL_c$ on the constant vectors $\vect{1}{0}$, $\vect{0}{1}$ is trivial, see \eqref{Lc0inv} and Remark \ref{Lc on V0}.
In order to diagonalize  $\cL_c$ on the zero average functions, we conjugate it  with the  symplectic map
\begin{equation}\label{M}
		 \mathcal{M}:=\left(
			\begin{matrix}{}
				M		& 	  0\\
				0 		&   M^{-1}
			\end{matrix}
		\right)\quad,\quad
		 M:=\left(\frac{G(0)}{g -\kappa\pa_x^2-\frac{\gamma^2}{4}\partial_x^{-1}G(0)\partial_x^{-1}} 
			\right)^{\frac{1}{4}} \ , 
	\end{equation} 
 i.e. 
	$\mathcal{M}^\top J \mathcal{M}= J$, 
     obtaining the Hamiltonian operator
$$
\cL_c^{(1)}:= \cM^{-1} \cL_c \cM = c \pa_x  + J \left(\begin{matrix}{}
				\omega(D)			& 	-\frac{\gamma}{2} \partial_x^{-1}G(0)\\
				\frac{\gamma}{2}G(0)\partial_x^{-1}	&  	\omega(D)
			\end{matrix}
		\right)
$$
	where $\omega (D)$ is the Fourier multiplier with symbol, for any $ \xi \neq 0 $,  
	$$
		\omega(\xi ):=
  \begin{cases}
			\sqrt{\big(g+\kappa \xi^2)
   \xi \tanh(\tth \xi )    +\frac{\gamma^2}{4} \tanh^2 (\tth \xi ) }	\quad  \ \text{if} \ \tth < + \infty \\
   \sqrt{\big(g+\kappa \xi^2)
   |\xi|     +\frac{\gamma^2}{4} } \qquad \qquad \qquad \quad 
   \qquad   
   \text{if} \ \tth = + \infty \, .  
   \end{cases}
	$$
Next we pass to complex coordinates via the invertible transformation
	\begin{equation}\label{defC}
		\mathcal{C} :=\frac{1}{\sqrt2}\left(
			\begin{matrix}{}
				1	  & 	 1\\
				-\im    &  \im
			\end{matrix}
		\right) \, , 
  \quad\mathcal{C}^{-1}
  =\frac{1}{\sqrt2}\left(
			\begin{matrix}{}
				1	  & 	 \im  \\
				1    &  - \im
			\end{matrix}  
		\right) \, , 
	\end{equation}
 obtaining the complex Hamiltonian operator
 \begin{equation}\label{L2c}
     \cL_c^{(2)}:= 
   \cC^{-1}   \cL_c^{(1)} \cC
     =  
     c \pa_x + J_c \boldsymbol{\Omega} 
     \quad \text{where} \quad 
      J_c:= \begin{pmatrix}
- \im & 0 \\
0 & \im 
 \end{pmatrix}
 \end{equation}
is the complex Poisson tensor, 
and  $\boldsymbol{\Omega}$ is the 
selfadjoint operator
 \begin{equation}\label{boldomega}
		\boldsymbol{\Omega}:= 	\begin{pmatrix}
				\Omega(D)	&	0\\
				0	&	\overline{\Omega(D)}
    \end{pmatrix}
        \quad \text{where}
        \quad \Omega(D):=  \omega(D)+\im \frac{\gamma}{2}G(0)\partial_x^{-1}
    \end{equation}
is the Fourier multiplier 
  with real symbol 
 $\Omega_\xi  :=
  \Omega_\xi (g, \tth, \kappa, \gamma) $ defined in \eqref{Omega}.
 In \eqref{boldomega} the operator 
$\overline{\Omega(D)} $ is the Fourier multiplier with symbol $ \Omega_{-\xi}$.

Note that if the surface tension $ \kappa > 0 $ then the dispersion relation $ \xi \mapsto \Omega_\xi $ in \eqref{Omega} 
is a symbol of order $3/ 2 $ with  asymptotic expansion
\begin{equation}\label{expOmega}
\Omega_\xi = \sqrt{\kappa}
| \xi |^{3/2} + a_0 (\xi) 
\end{equation}
where $a_0(\xi) $ is a Fourier multiplier of order $0 $. If $ \kappa = 0 $ then 
$ \Omega_\xi $ in \eqref{Omega} 
is a symbol of order $ 1/ 2 $. 

In the Fourier basis 
$ \big\{e^{ \im  j x}\big\}_{j \in \Z \setminus \{0\} } $, the operator $\cL_c^{(2)}$ in 
\eqref{L2c} acts as the diagonal operator  
\begin{equation}\label{cLj}
    \cL_c^{(2)} = \textup{diag}(\cL_j)_{j \in \Z \setminus \{0\}} \quad \text{where}\quad \cL_j:= \begin{pmatrix}
        \im (c j - \Omega_j) & 0 \\
        0 & \im (c j +\Omega_{- j}) 
    \end{pmatrix}   \, , \ 
    \forall  j \in \Z \setminus \{0\}  \, .  
\end{equation}
Recalling \eqref{dFc0}, the possible speeds of  bifurcation $c \in \R $ are those for which  $\cL_c^{(2)}$ has a nontrivial kernel, namely, in view of \eqref{cLj}, the set 
$ \mathsf C $ defined in \eqref{bifspeed}.
\\[1mm]
\noindent 
{\bf Decomposition of the phase space in symplectic  subspaces invariant under $ \cL_c $.} 
We decompose any function 
$ (\eta, \zeta)  $
of  the  space 
$ L^2 (\T) \times L^2 ( \T ) $
as
\begin{equation}\label{deco-real}
u (x) = \begin{pmatrix}
\eta (x) \\
\zeta (x) 
\end{pmatrix}  = \eta_0 
v_0^{(1)} 
+
\zeta_0 v_0^{(2)}
+ \sum_{j \in \Z\setminus\{0\}}  
\alpha_j v_j^{(1)}(x) + \beta_j v_j^{(2)}(x) 
\end{equation}
where, denoting $M_j  $ the real symbol of the Fourier multiplier in \eqref{M}, 
\begin{equation}\label{vv}
v_0^{(1)} := \vect{1}{0} \, ,\ v_0^{(2)} := \vect{0}{1} \, , 
\quad 
 v_j^{(1)}(x) := \vect{M_j \cos(jx)}{M_j^{-1}\sin(jx)} \, , \ 
  v_j^{(2)} (x) := \vect{-M_j \sin(jx)}{M_j^{-1}\cos(jx)} 
  \, , 
\end{equation}
for any  $ j \in \Z \setminus \{0\} $,  
and
\begin{equation}\label{coordinatev}
\alpha_j :=  \alpha_j (u) =  \cW (u,v_j^{(2)}) \, , \quad 
\beta_j :=  \beta_j (u) = - \cW (u,v_j^{(1)}) \, . 
\end{equation}
The 
$2$-dimensional 
real vector spaces
\begin{equation}\label{defVj}
  V_0 
:= \la v_0^{(1)}, v_0^{(2)} 
\ra   \, , 
\quad V_j := \left\la  v_j^{(1)} ,v_j^{(2) } \right\rangle \, ,  
\  j \in \Z \setminus \{ 0\} \, ,  
\end{equation}
are  invariant under 
the action of $ \cL_c $, 
\begin{equation}\label{cLVj}
\cL_c : V_j \to V_j \, , \quad \forall j \in \Z \, , \quad \forall c \in \R \, , 
\end{equation}
 as
\begin{equation} \label{Lc0inv}
\cL_c v_0^{(1)}
= 
- g v_0^{(2)} \, ,  
\end{equation}
and, for any $ j \neq 0 $, 
 \begin{equation}\label{cLincl}
 \begin{aligned}
&  \cL_c v_j^{(1)} = 
(c j - \Omega_j)  
v_j^{(2)} \, ,\quad  
\cL_c v_j^{(2)}  =  
- (c j - \Omega_j) v_j^{(1)}\, , \\
& 
\qquad \quad \ 
\pa_x v_j^{(1)} = j v_j^{(2)} \, ,  \ \quad 
\pa_x v_j^{(2)}= - j v_j^{(1)}  \, .
\end{aligned}
\end{equation}
The first identities in \eqref{cLincl} directly follow by \eqref{L2c} and
\eqref{cLj} since  
each subspace $V_j $ in \eqref{defVj} is, for any $ j \neq 0 $,  the image under 
the  map $\cM \cC $ defined in \eqref{M}, \eqref{defC} of the $ 1$-d complex vector subspace  
$ \vect{ z e^{\im j x}}{ \overline{z} e^{-\im j x}} $, $ z \in \C $,  namely 
\begin{equation} \label{VjMCzj}
v_j^{(1)}(x) =
\cM \cC  \frac{1}{\sqrt{2}}\vect{e^{\im j x}}{e^{-\im j x}} \ , \quad 
v_j^{(2)}(x)=\cM \cC \frac{1}{\sqrt{2}}\vect{\im e^{\im j x}}{-\im e^{-\im j x}} \, . 
\end{equation}

\begin{remark}\label{Lc on V0}
We have $ \cL_c v_0^{(2)} = 0 $
but the vector
$ v_0^{(2)} := \vect{0}{1} $ does not belong to
ker$(\cL_c^*)$ because 
$ v_0^{(2)} \notin X $ (the space $ X $ is introduced in   
\eqref{spaceX}) since the second component of $ v_0^{(2)} $ has non-zero average.
\end{remark}

Each subspace $ V_j $ is symplectic 
since the symplectic form $\cW $
in \eqref{W} restricted to $ V_j $ reads 
\begin{equation}\label{2formW}
 \cW_{|V_j} \equiv 
 \di \alpha_j \wedge \di \beta_j\,  \qquad \text{as} \qquad
	\mathcal{W}(v_j^{(1)},v_j^{(2)})
            = 1  \, , \quad \forall j \in \Z \, . 
\end{equation}
Furthermore  the subspaces $ V_j $ 
are pairwise 
symplectic orthogonal, namely 
\begin{equation*}     V_k
\bot^{\cW} V_j  \, ,  \qquad \forall  k\not=j \, ,
\end{equation*} 
but they are {\it not} all orthogonal, for example 
$ V_{j} $
and $ V_{-j}$. 

The spaces $V_j $ are invariant with respect 
the involution $\mathscr{S}$  in \eqref{sym}
and the translations $\tau_\theta$
in \eqref{sym tau}, namely  		
$ \mathscr{S} V_j \subset V_j $,
$ \tau_\theta V_j\subset V_j $.
Since $\mathscr{S} v_j^{(1)} = v_j^{(1)} $ and
$\mathscr{S} v_j^{(2)} = - v_j^{(2)} $, 
and
$$
\tau_\theta v_{j}^{(1)}
= \cos (j \theta)  v_{j}^{(1)} 
- \sin (j \theta)  v_{j}^{(2)} \, , \quad 
\tau_\theta v_{j}^{(2)}
= \sin (j \theta)  v_{j}^{(1)} 
+ \cos (j \theta)  v_{j}^{(2)} \, 
, 
$$
in the coordinates $(\alpha_j, \beta_j )$,  they read 
		\begin{equation}\label{SRtheta}
\mathscr{S}:\left(\begin{matrix}\alpha_j	\\	\beta_j\end{matrix}\right)\mapsto 
			\left(\begin{matrix}{}
						1	& 	0\\
						0	&	-1
					\end{matrix}\right)
			\left(\begin{matrix}\alpha_j	\\	\beta_j \end{matrix}\right)
			\, ,  \quad
   \tau_\theta:\left(\begin{matrix}\alpha_j	\\	\beta_j\end{matrix}\right)\mapsto 
   \underbrace{\left(\begin{matrix}{}
	\cos( j \theta)	&
                                        \sin(j \theta)\\
      - \sin( j \theta)	&	
                            \cos(j \theta)
					\end{matrix}\right) 
                }_{=: R(- j\theta)}
			\left(\begin{matrix}\alpha_j	\\	\beta_j\end{matrix}\right) 
		\end{equation}
where $ R(\alpha) $ is the clock-wise 
rotation matrix of angle 
$ \alpha $. 

\begin{remark}\label{O2action}
The action of the semi-direct product group $ {\mathbb S}^1 \rtimes \mathbb{Z}_2$ on
$ V_j $ defined  by 
$ (\tau_\theta)_{\theta \in {\mathbb S}^1} $  and $ \{ I, \mathscr{S} \}$
in \eqref{SRtheta} amounts to the action of the orthogonal group $ O(2) $,
because each $ 2 \times 2 $ orthogonal matrix can be written as the composition of
a rotation matrix and a reflection.
\end{remark}

    Writing $u $ as in \eqref{deco-real} the quadratic part of the Hamiltonian $\mathcal{H}$ in \eqref{EqWWZakharovHamNewCoordinates} is equal to 		\begin{equation}\label{HamiltonianInCoord}
            \cH_2 (u) := \frac{1}{2} \langle \di_u \nabla\mathcal{H}(0)u,  u\rangle = \frac12 
     \sum_{j\not=0}\Omega_j (\alpha_j^2+\beta_j^2) +  \,\frac{ g \eta_0^2}2
    \end{equation}
     and the momentum 
     \begin{equation}\label{MomentumInCoord}
			\mathcal{I}(u) = 
			\frac{1}{2}\langle \di_u \nabla\mathcal{I}(0)u,  u\rangle = 
			-  \frac12 \sum_{j \not=0}j (\alpha_j^2+\beta_j^2) \, . 
    \end{equation}
{\bf The kernel of $\cL_c$.}
We now fix some  $j_* \in \Z\setminus \{0\}$ and 
consider the bifurcation speed 
$$ 
c_*= \frac{\Omega_{j_*}}{j_*} \in \mathsf C \, , 
$$
cfr. \eqref{def:c*}, \eqref{bifspeed}.
Clearly, in view \eqref{cLincl}, 
the subspace 
$ V_{j_* } $ in \eqref{defVj} 
is included in the Kernel
of $ \cL_{c_*} $. It remains to
understand if there are other 
eigenvalues $ \cL_{c_*} $ equal to zero.
In view of the previous analysis 
\begin{equation}\label{kernelL}
\ker(\cL_{c_*}) = \bigoplus_{j \in \cV} V_j \, ,
\quad 
\cV := 
\big\{ j \in \Z \setminus \{0\}\, , \
   {\Omega_j = c_* j} \big\} \, , 
\end{equation}
and the problem is reduced 
to determine the existence of integers $ j $ that solve the equation  
\eqref{omegainter}. 
Since $\Omega_j>0$ for any $j\in \mathbb{Z}\setminus\{0\}$   the integers 
$j_*$ and $j$ satisfying \eqref{omegainter}
have the same sign
\begin{equation}\label{segni}
    \sign(j)=\sign(j_*) \, .
\end{equation}
If there are no other integer $ j \neq j_*  $ satisfying \eqref{omegainter} 
then the kernel of $ \cL_{c_*} $ is  
		 $2$ dimensional (non-resonant case). 
   Otherwise, as we prove in the next section, 
   the kernel  of  
$    \cL_{c_*} $ is  $4$ dimensional (resonant-case).

\section{The dispersion relation}\label{sec:dispe}

In this  section we prove Proposition \ref{KER24}. 
We study the graph of the function 
$ f : \R \setminus \{0\} \to \R $, $ \xi \mapsto 
f(\xi) := f(\xi;g,\tth,\kappa,\gamma) $ defined,
if $ \tth < + \infty $, by 
    \begin{equation}\label{fxdef}
           f(\xi ) :=  \frac{\gamma}{2}\frac{\tanh(\tth \xi )}{\xi}+ 
                    \text{sign}(\xi)\sqrt{\Big(g+ \kappa \xi^2+\frac{\gamma^2}{4} \frac{\tanh(\tth \xi)}{\xi}\Big) \frac{\tanh(\tth \xi)}{\xi}} 
    \end{equation}
    and, in infinite depth  $\tth = + \infty $, by \begin{equation}\label{fxdefinf}
        f(\xi) :=  \frac{\gamma}{2|\xi|}+ 
                \text{sign}(\xi)\sqrt{\Big(g+ \kappa \xi^2+\frac{\gamma^2}{4|\xi|} \Big) \frac{1}{|\xi|}} \, , 
    \end{equation}
that, for any 
$j \in \Z \setminus \{0\} $,
is equal to the dispersion relation $f( j):=\frac{\Omega_{j}}{ j}$,
cfr. \eqref{Omega}.

Note  that 
$ f(\xi;g,\tth,\kappa,\gamma)=-f(-\xi;g,\tth,\kappa,-\gamma) $ 
and so it is sufficient to study the function $f(\:\cdot\:;g,\tth,\kappa,\gamma)$ for $\gamma\geq0$.

\begin{proposition}\label{paravari}
For any  $ g \in
(0,+\infty)$, $\kappa\in[0,+\infty)$,
$ \tth \in (0,+\infty] $,
$ \gamma \in [0,+\infty) $,
the  functions in \eqref{fxdef}-\eqref{fxdefinf} satisfy 
\begin{equation}\label{fpm}
f(\xi) > 0 \, , \ \forall  \xi > 0 \,  \qquad \text{and} \qquad  
f(\xi) < 0  \, , \ \forall \xi < 0 \, .
\end{equation} 
According to the values of the parameters the following properties hold:
\\[1mm]
$(1)$ \textsc{Finite depth case} $\tth< + \infty $.
\begin{itemize}
    \item[(1a)] \underline{$ \kappa > 0 $.}
Then the function $f$ fulfills
\begin{align}\label{fpmK1}
&  \lim_{\xi \to 0^\pm} f(\xi)=\frac{\gamma}{2}\tth\pm\sqrt{\big(g+\frac{\gamma^2}{4}\tth\big)\tth} \ , \qquad
\lim\limits_{\xi \to  \pm\infty}f(\xi) = \pm \infty \, , 
\quad \lim_{\xi \to 0^\pm} f'(\xi) = 0 \, , \\
& \label{fpmK2}
\lim_{\xi \to 0^\pm}
        f''(\xi) 
        = \pm \alpha \Big( B_\pm - \frac13 \Big) \quad \text{where} \quad 
        \alpha := 
        \frac{2 g \tth^3}{\sqrt{\big(4g+\gamma^2\tth\big) \tth}}
\end{align}
and $ B_\pm $ are the vorticity-modified Bond numbers in \eqref{def:bond}. 
Moreover
\begin{equation}\label{casipm}
\begin{aligned}
&	\lim_{\xi \to 0^+}f''(\xi;g,\tth,\kappa,\gamma)\geq0\iff f_{\vert(0,+\infty)}\text{ is  strictly increasing} \, ,\\
& 	\lim_{\xi \to 0^+}f''(\xi;g,\tth,\kappa,\gamma)< 0\iff f_{\vert(0,+\infty)}\text{  has a unique local minimum} \, , \\
	& \lim_{\xi \to 0^-}f''(\xi;g,\tth,\kappa,\gamma)\leq0\iff f_{\vert (-\infty,0)}\text{ is  strictly increasing}\, ,\\
 &	\lim_{\xi \to 0^-}f''(\xi;g,\tth,\kappa,\gamma)>0\iff f_{\vert (-\infty,0)}\text{ has a unique local maximum} \, . 
\end{aligned}
\end{equation}
\item[(1b)] \underline{$ \kappa = 0 $.} The functions
$f_{\vert(0,+\infty)}$ and $f_{\vert(-\infty,0)}$ are strictly decreasing and
\begin{equation}\label{limitaltro}
\lim_{\xi\to 0^\pm} f(\xi) =\frac{\gamma}{2}\tth\pm\sqrt{\big(g+\frac{\gamma^2}{4}\tth\big)\tth} \ , 
  \quad 
    \lim_{\xi\to\pm \infty}f(\xi)=0^\pm \ , 
    \quad 
    \lim_{\xi \to 0^\pm} f'(\xi) = 0 \ . 
\end{equation}
\end{itemize}
$(2)$ \textsc{Deep water case} $\tth = + \infty $.
\begin{itemize}
    \item[(2a)]  \underline{ $ \kappa > 0 $.}  Then  $f_{\vert(0,+\infty)}$ has a unique local minimum, $f_{\vert (-\infty,0)}$ has a unique local maximum and
\begin{equation}\label{fpm1:infdepth}
\begin{aligned}
    \lim\limits_{\xi \to 0^+} f(\xi)  = +\infty & \ , \quad 
    \lim\limits_{\xi \to 0^-} f(\xi)  = 
    \left\{\begin{array}{@{}l@{}l}
    -\frac{g}{\gamma} &\quad\text{if }\gamma>0\\
    -\infty           &\quad\text{if }\gamma=0
    \, , 
    \end{array}\right. \\
    \lim\limits_{\xi \to \pm \infty} f(\xi)  = \pm\infty \,, & 
    \quad \lim\limits_{\xi \to 0^-} f'(\xi)  = \left\{\begin{array}{@{}l@{}l}
    -\frac{g^2}{\gamma^3} &\quad\text{if }\gamma>0\\
    -\infty               &\quad\text{if }\gamma=0 \, .  
    \end{array}\right. 
\end{aligned}
\end{equation}
 \item[(2b)] \underline{ $ \kappa = 0 $.}   
The functions $f_{\vert(0,+\infty)}$ and $f_{\vert(-\infty,0)}$ are strictly decreasing and
\begin{equation}\label{fpm2:infdepth}
\begin{aligned}
    \lim\limits_{\xi \to 0^+} f(\xi)  = +\infty & \ , \quad 
    \lim\limits_{\xi \to 0^-} f(\xi)  = 
    \left\{\begin{array}{@{}l@{}l}
    -\frac{g}{\gamma} &\quad\text{if }\gamma>0\\
    -\infty           &\quad\text{if }\gamma=0 \, , 
    \end{array}\right. \\
    \lim\limits_{\xi \to \pm \infty} f(\xi)  = 0^\pm \ & ,
    \quad \lim\limits_{\xi \to 0^-} f'(\xi)  = \left\{\begin{array}{@{}l@{}l}
    -\frac{g^2}{\gamma^3} &\quad\text{if }\gamma>0\\
    -\infty               &\quad\text{if }\gamma=0 \, . 
    \end{array}\right. 
\end{aligned}
\end{equation}
\end{itemize}
\end{proposition}

\begin{figure}[h]\label{fig:dis}
    \begin{center}
    \includegraphics[ width =0.3\textwidth ]{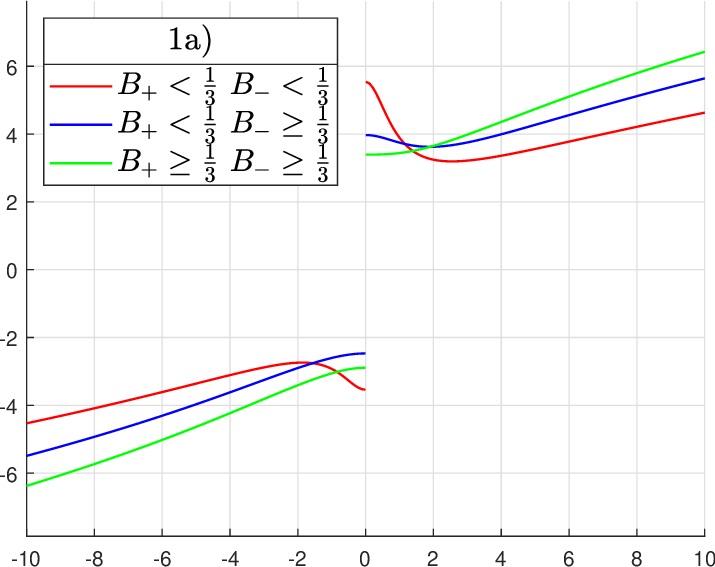}
    \hspace{0.05\textwidth}
    \includegraphics[ width =0.3\textwidth ]{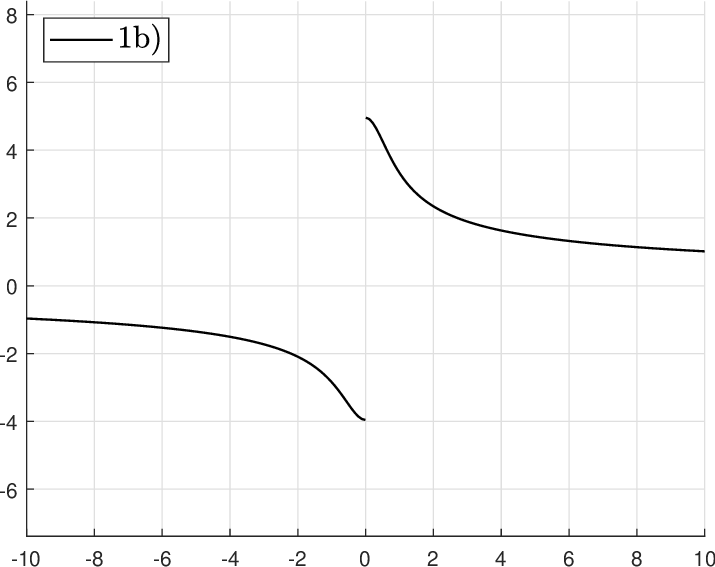}
    \vspace{0.05\textwidth}\\
    \includegraphics[ width =0.3\textwidth ]{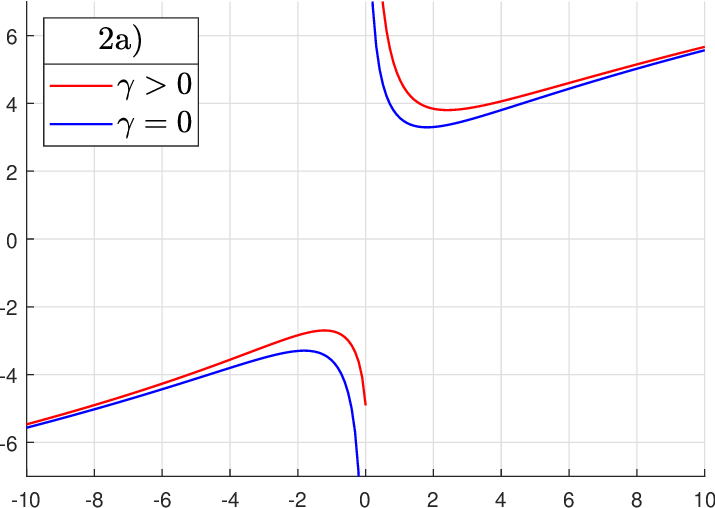}
    \hspace{0.05\textwidth}
    \includegraphics[ width =0.3\textwidth ]{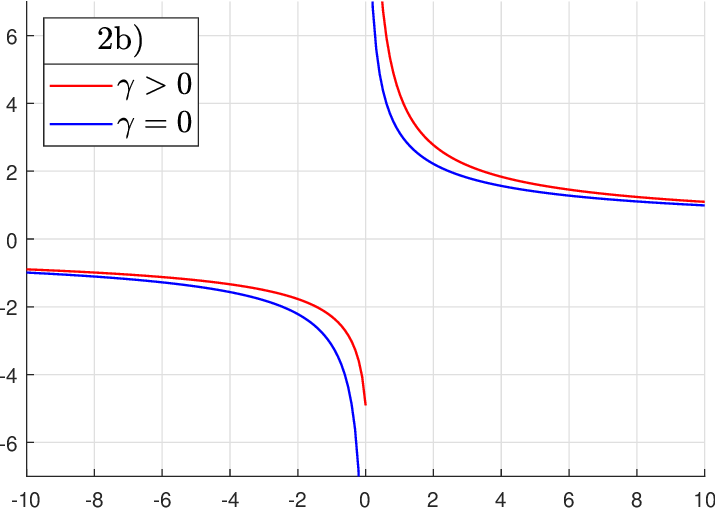}
    \end{center}
    \caption{Graphs of $ f(\xi) $ according to different values of
    $(g,\tth, \kappa, \gamma) $.
    The function $ f $ 
    has a minimum in $ (0,+\infty)$ iff $ B_+ < 1/3 $ or in case (2a), and have a maximum in $ (-\infty,0)$ iff $B_-<1/3$ or in case (2a).}
    \label{figure:Bond} 
\end{figure}

Proposition \ref{KER24} is a direct consequence of  Proposition \ref{paravari}.

\begin{proof}[Proof of Proposition \ref{KER24}]
In case (1a), by \eqref{fpm}, \eqref{fpmK1}, \eqref{fpmK2}, \eqref{casipm}
the graph of the function $ f (\xi) $ 
has the forms described in 
Figure \ref{figure:Bond}a) according to the values of $ B_\pm $. 
In particular $ f(\xi) $ has a unique absolute minimum for $ \xi > 0 $ if and only if 
 $ B_+ < 1/ 3 $ and a unique absolute maximum for $ \xi < 0 $ if and only if 
 $ B_- < 1/ 3 $.
 Note that if $\gamma\geq 0$ then $B_+ \leq B_- $, cfr.
    \eqref{def:bond}, and so the case $B_+ \geq 1/ 3 $ and $ B_- < 1/ 3 $ never happens.
 In the other cases 
 the graphs of the function $ f (\xi) $ have the form in Figure \ref{figure:Bond}b), 2a),2b). 
 The existence of integers $ j \neq  j_* $ such that \eqref{omegainter} is 
 possible   only  in cases 1a)-2a). 
For any $ 1 \leq j_* < j $ the function
$ f(j;g,\tth,\kappa,\gamma) - f(j_*; g,\tth,\kappa,\gamma) $ tends to $ + \infty $ as $ \kappa \to + \infty $
and tends to $ f(j;g,\tth,0,\gamma) - f(j_*; g,\tth,0,\gamma) < 0 $ as $ \kappa \to 0^+ $, which is negative  by cases 1b)-2b).
Proposition \ref{KER24} follows. 
\end{proof}

\begin{remark}
The vorticity-modified bond numbers in \eqref{def:bond} satisfy 
$ \lim_{\tth \to + \infty} B_+ = - \infty $  and  
$ B_- = \tfrac13 - \frac{g}{3 \gamma^2 \tth} + O( \tth^{-2}) $
as $ \tth \to + \infty $. Thus both $ B_+, B_- < \tfrac13  $
for sufficiently deep water and the graphs of cases 1a)-2a) look similar.  
Note also that 
$ -\frac{\tth\gamma^2}{6g}\big(1\pm 
\sqrt{ 1 + \frac{4g}{\tth \gamma^2}}
\ \big) < 1/ 3 $ for any $ g, \tth, \gamma $. 
\end{remark}

The rest of this section 
is dedicated to prove  Proposition \ref{paravari}.

Property \eqref{fpm} directly follows by \eqref{fxdef}-\eqref{fxdefinf}.
The function $ f(\xi) $ in \eqref{fxdef} admits an analytic extension at $ \xi = 0 $
and a Taylor expansion provides 
the limits of $ f(\xi )$  as $ \xi \to 0^\pm $ 
in \eqref{fpmK1}, \eqref{limitaltro} and  \eqref{fpmK2}.  
The limits of $ f(\xi )$ for $ \xi \to \pm \infty $ follow directly. 
 Also the limits  \eqref{fpm1:infdepth}-\eqref{fpm2:infdepth} of the function $ f(\xi) $
in  \eqref{fxdefinf} are directly verified.

In order to  prove \eqref{casipm}, the fact that in case (2a) the function $f_{\vert(0,+\infty)}(\xi) $
has a unique local minimum and that in cases (1b)-(2b) the function 
$f_{\vert(0,+\infty)}(\xi) $ is strictly decreasing (with the analogous properties for
$ f_{\vert(-\infty,0)}(\xi) $)  we first provide the following lemma, whose proof is related to \cite{M}.  
 We denote
\begin{equation}\label{fpiumeno}
    f_+(\xi):=f_{\vert(0,+\infty)}(\xi) \, , \quad 
    f_-(\xi):=f_{\vert(-\infty,0)}(-\xi) \, , 
    \quad \forall \xi >  0 \, . 
\end{equation}

    \begin{lemma}\label{CriticalValuesAreLocalMinima}
        If $\kappa=0$ the functions $f_\pm (\xi )$  have no critical points.  
        For any value of the parameters $(g,\tth,\kappa,\gamma)$ with $ \kappa > 0 $, 
        any critical point $ \bar \xi > 0 $ of $f_+ (\xi) $, resp. $f_- (\xi) $, is a strict local minimum, resp.  maximum, i.e. 
    \begin{equation}\label{der1der2}
        f'_\pm( \bar \xi)=0 \quad \implies \quad \pm f''_\pm( \bar \xi)>0 \, . 
    \end{equation} 
        The function 
        $f_+(\xi) $, resp. $f_-(\xi)$, 
        has at most one strict local minimum,  resp. maximum, in $[0, + \infty)$.
    \end{lemma}
    \begin{proof}
    We prove the deep water case $\tth=+\infty$. The case $\tth<+\infty$ is analogous and is 
    covered in \cite{M}. In view of \eqref{fxdefinf} we have, for any $ \xi > 0 $, 
\begin{equation}\label{deltaf}
            f_+(\xi)+f_-(\xi)=\frac{\gamma}{\xi}\, , \quad   
            f_+(\xi)f_-(\xi)=-\frac{g+\kappa \xi^2}{\xi} \, ,
        \end{equation}
        and thus $ f_\pm (\xi) $ solves the second order equation
        \begin{equation}\label{roots of polynomial}
             \xi f_\pm(\xi)^2- \gamma f_\pm(\xi)- (g+\kappa \xi^2) =0 \, ,  \quad \forall \xi > 0 
             \, . 
        \end{equation}
        The reciprocal functions 
       \begin{equation}\label{defgf}
            g_+(\xi):=\frac{1}{f_+(\xi)} > 0 \, ,\quad g_-(\xi):=\frac{1}{f_-(\xi)} < 0 
            \, , \quad \forall \xi > 0  \, ,  
        \end{equation}
        are well defined
         since $f_+(\xi) > 0$ and $f_-(\xi)<0$ for any $ \xi > 0 $ by \eqref{fpiumeno}, \eqref{fpm},       and   
        solve, in view of  \eqref{roots of polynomial},   
        \begin{equation}\label{roots of polynomial g}
            (g+\kappa \xi^2) g_\pm^2(\xi)+\gamma  g_\pm(\xi)=\xi \, , \quad \forall \,   \xi > 0 \, .
            \end{equation}
     The critical points of $g_\pm$ are  critical points of $f_\pm$ and viceversa.     
        Differentiating \eqref{roots of polynomial g} in $ \xi $ 
        we get, for any $\xi > 0 $, 
        \begin{equation}\label{D eq}
         2\kappa\xi g_\pm^2 (\xi) + 2(g+\kappa\xi^2)g_\pm(\xi)g_\pm'(\xi)+ \gamma g_\pm'(\xi)  = 1 \, .
        \end{equation}
        For $\kappa=0 $ the function $f_\pm (\xi)  $ has no critical points.
        Indeed, if $ f'_\pm (\overline{\xi} ) = 0  $ 
         then $ g'_\pm (\overline{\xi} ) = 0 $ and 
         by \eqref{D eq}  we get 
        $ 0 = 1 $.
        This contradiction proves the first statement of the lemma.
        
        Thus in the sequel we assume  $\kappa>0$.
        
        Differentiating  \eqref{D eq} at $ \xi=\overline{\xi}$ we get 
                   $  \big( 2(g+\kappa{\bar\xi}^2)g_\pm(\bar\xi)  + \gamma  \big)g_\pm''(\bar\xi)
                   =  - 2\kappa g_\pm^2(\bar\xi) $
               and using \eqref{roots of polynomial g}  
                we get
               \begin{equation}\label{profin}
                    \Big((g+\kappa\bar\xi^2)g_\pm(\bar \xi)+\frac{\bar \xi}{g_\pm(\bar\xi)}\Big)g''_\pm(\bar\xi)=-2\kappa g^2_\pm(\bar\xi)\, .
               \end{equation} 
        Since $ g_+ (\overline{\xi} ) > 0 $ (cfr. \eqref{defgf}) we deduce by
        \eqref{profin} that $ g_+''
                   (\overline{\xi}) <  0 $ 
                   and then
                   $ f_+ ''(\overline{\xi}) =-\frac{g_+ ''(\overline{\xi})}{g_+^2 (\overline{\xi})} > 0 $ and thus $ f_+ ''(\overline{\xi}) > 0 $. 
                   Similarly, since $ g_- (\overline{\xi} ) < 0 $ we deduce 
                   $ g_-''
                   (\overline{\xi}) >  0 $ and thus $ f_- ''(\overline{\xi}) < 0 $. 
               This proves
                \eqref{der1der2}.
        
               Let us prove the last claim of the lemma. 
                If $f_+ (\xi) $  has two distinct local minima $ 0 \leq \xi_1<\xi_2 $, then 
                $f_+ (\xi) $ has a  
                local maximum point in $(\xi_1,\xi_2)$. This is a contradiction with \eqref{der1der2}. 
                The claim for $ f_- (\xi) $
                follows similarly. 
\end{proof}

The next lemma 
allows to deduce
that a $ 4 $ dimensional kernel of $ {\cal L}_{c_*} $ 
may occur if and only if $ B_{\pm} < 1/3 $.

 \begin{lemma}
\eqref{casipm} holds. 
 \end{lemma}
\begin{proof}
    We prove the first equivalence in \eqref{casipm}.
    The other ones are analogous. 
    Lemma \ref{CriticalValuesAreLocalMinima} implies that 
    $f_+(\cdot;g,\tth,\kappa,\gamma)$ is increasing in $ [0,+\infty) $ 
    if and only if  it is {\it strictly} increasing.
    Therefore \eqref{casipm} is equivalent to the following claim.
    \\[1mm]
    {\bf Claim:} {\it 
    $\lim_{\xi\to 0^+} f_+''(\xi;g,\tth,\kappa,\gamma)\geq0$ if and only if
    $\xi\mapsto f_+(\xi;g,\tth,\kappa,\gamma)$ is increasing.}
    
    Since $ \lim_{\xi\to 0^+} f_+'(\xi;g,\tth,\kappa,\gamma) = 0 $ by 
    \eqref{fpmK1}, the implication $\Leftarrow$ is trivial.
    To prove the other implication we consider two different cases:
    
    {\it Case 1: }$\lim_{\xi\to 0^+} f_+''(\xi;g,\tth,\kappa,\gamma)>0$. 
    Then $ f_+'(\xi;g,\tth,\kappa,\gamma)>0 $ for any $ \xi > 0 $  
    by  Lemma \ref{CriticalValuesAreLocalMinima}. 
    Thus  $\xi\mapsto f_+(\xi;g,\tth,\kappa,\gamma)$ is increasing.
    
    {\it Case 2: }  $\lim_{\xi\to 0^+} f_+''(\xi;g,\tth,\kappa,\gamma)=0 $. 
    By \eqref{fpmK2} and \eqref{def:bond} we have
        $$
            0=\lim_{\xi \to 0^+}  f_+''(\xi;g,\tth,\kappa,\gamma)= \alpha
            \Big(\frac{\kappa}{g\tth^2} - \frac{\gamma}{6g}\sqrt{\big(4g+\gamma^2 \tth\big)\tth} - \frac{\tth\gamma^2}{6g}  - \frac13 \Big) \quad \text{with} \quad \alpha>0\, .
        $$
    Therefore, for any $\varepsilon>0$, $\lim_{\xi \to 0^+}  f''(\xi;g,\tth,\kappa+\varepsilon,\gamma)>0$ and thus, by case 1, the function $\xi \mapsto f_+(\xi;g,\tth,\kappa+\varepsilon,\gamma)$ is increasing in $ \xi \in [0,+\infty) $. Therefore the 
    limit function $ f_+(\xi;g,\tth,\kappa,\gamma) $
    as $ \varepsilon \to 0^+ $ is increasing as well in $ \xi \in [0,+\infty) $.
    The claim is proved.
\end{proof}

By Lemma \ref{CriticalValuesAreLocalMinima}
    in case (2a) the function $ f_{\vert(0,+\infty)}(\xi) $
has a unique local minimum.  
In cases (1b)-(2b) the function 
$f_{\vert(0,+\infty)}(\xi) $ is strictly decreasing 
still  by Lemma \ref{CriticalValuesAreLocalMinima}. 
The proof of Proposition \ref{paravari} is complete.

\section{Variational Lyapunov-Schmidt reduction}\label{sec:LS}

We decompose the phase space 
$ L^2 \times L^2 := L^2 (\T, \R) \times L^2 (\T, \R) $ 
equipped with the symplectic form $ \cW $ in \eqref{W}
as
		\begin{equation}\label{L2L2VW}
			L^2  \times L^2  = V \oplus W
		\end{equation}
where, recalling \eqref{dFc0},
\eqref{dNablaI} and \eqref{kernelL}, 
  \begin{align}
V & := \ker(\cL_{c_*})
  = \big\{ c_* \pa_x v +
  J \di_u \grad  \cH(0) v = 0 \}  =
  \big\{  
   \di_u \grad  (\cH + c_* \cI)  
  (0) v = 0 \} \label{primoV}  \\     
  & = 
  \bigoplus_{j \in \cV } V_j 
  \, , \label{defV} \\
  \label{Def W}
&   W := V^{\bot_{\cW}} := 
V_0 \oplus \mathscr{W}  
\, , \qquad 
\mathscr{W}:=\overline{\bigoplus_{j \in \Z \setminus \{0\},  
j \notin \cV} V_j}^{L^2\times L^2}, 
\end{align}
and $ V_j $,  $ j \in \Z $,  are the bi-dimensional symplectic subspaces defined in \eqref{defVj}. 
The subspaces $V$ and $W$ are symplectic and each one is the symplectic orthogonal of the other.
We denote by $ \Pi_V $ and $ \Pi_W $
the {\it symplectic} projectors on $ V $, respectively $ W $,
induced by the decomposition \eqref{L2L2VW}. 
Since $V$ and $W$ are symplectic orthogonal,  
the projectors $ \Pi_V $, $ \Pi_W $ satisfy
\begin{equation}\label{Symplectic Projectors}
    \cW(\Pi_V u,u_1)=\cW(u, \Pi_V u_1)\, , \quad 
    \cW(\Pi_W u,u_1)=\cW(u, \Pi_W u_1)\, , \quad
    \forall u,u_1 \in L^2 \times L^2 \, . 
\end{equation}
In order to solve
\eqref{eq:H} we implement a symplectic Lyapunov Schmidt-reduction.
According to  \eqref{L2L2VW} 
the space $X$ defined in \eqref{spaceX} and the  target space $Y$ defined  in \eqref{spaceY}
admit the 
decomposition  
\begin{equation}\label{XYVW}
        X = V\oplus (W\cap X) \, , \quad
        Y = V\oplus (W\cap Y) \, ,
\end{equation}
in 
symplectic orthogonal subspaces.  
We denote by 
$\Pi_{W\cap X} $ and 
$\Pi_{W\cap Y} $ the symplectic projectors  on $W\cap X$ and $W\cap Y $ induced by  \eqref{XYVW}. 
Decomposing uniquely $u \in X $
as  $ u = v + w $ with 
 $ v\in V$ and $w\in W\cap X$, 
 the equation \eqref{eq:H} is then equivalent to the system 
	\begin{equation}\label{RangeBifurcation}
		\left\{\begin{array}{@{}l@{}l}
			\Pi_V \cF(c, v+w)=0 \\
		      \Pi_{W\cap Y}\cF(c, v+w) =0  \, . 
		\end{array}\right. 
	\end{equation}
We call the first equation the bifurcation equation and the second one the range equation.  
We now solve the range equation by means of the implicit function theorem. 

    We denote by $B_r^V(0) $ the ball of radius $r$ and center 0 in $V$.
\begin{lemma}
{\bf (Solution of range equation)}\label{range equation}
		There exists  an analytic function 
		$w: B_r (c_*)
  \times B_r^V (0)
  \subset \R \times
  V 
  \to W\cap X$ 
  defined in a neighborhood   of $(c_*,0)$  satisfying 
  \begin{equation}\label{dw0}
  w(c,0) = 0 \, ,  
   \  \quad \di_v w (c,0) = 0 \, , 
   \quad\forall c \in B_r(c_*) \, ,
  \end{equation}
  such that 
\begin{equation}  \label{range}
	\Pi_{W\cap Y} \cF(c, v+w(c,v)) =0 \, . 
\end{equation}
The function $w(c, \cdot) $ is equivariant with respect to
the involution $\mathscr{S}$ and 
the translations $\tau_\theta$, namely 
	\begin{equation}\label{invarL}  
			w(c,\mathscr{S} v)=\mathscr{S} w(c, v)\, , \quad 
   w(c, \tau_\theta v)=\tau_\theta w(c,v) \, , \quad \forall v\in B_r^V(0) \, , \quad 
   \forall \theta\in \mathbb{R}  \, . 
		\end{equation}
	
	\end{lemma}

 \begin{proof}
	We apply the implicit function theorem 
to
$$
G : \R \times V \times
(W\cap X) \to (W\cap Y) \, , \  (c,v,w) \mapsto 
G(c,v,w) := \Pi_{W\cap Y} \cF(c, v+w) \, .
$$
The map $ G $ is analytic in a small neighborhood of 
$ (c_*,0,0) $ using  Theorem 1.2 in \cite{BMV2} 
about the analyticity  of the Dirichlet-Neumann operator
and the algebra properties of $H^{\sigma,s}(\T)$ 
which holds for $ s > 1/ 2 $. 
The condition $ s \geq \tfrac72 $ implies that $ s - \tfrac32 \geq s_0 \geq \tfrac{d+1}{2} $, 
where $ d = 1 $ 
and $ s_0 = 2  $
according to the notation of \cite{BMV2}[Theorem 1.2],
which also requires
$\frac{s+1}{2}\in \N $.
As commented below \eqref{spaceX} this is not restrictive for the results of the paper.
Theorem 1.2 of \cite{BMV2} is proved in the more delicate deep water case 
$ \tth = + \infty $ but its proof also holds if $ \tth < + \infty $, 
in this case see also \cite{NR}.
It results 
\begin{equation}\label{Gc00}
G(c,0,0) = 0 \, , \quad \forall c \in \R \, ,
\end{equation}
and we claim that
\begin{equation} \label{dG0}
\di_w G(c_*, 0,0 ) = \Pi_
{W\cap Y}  \di_u \cF(c_*,0)=\Pi_{W\cap Y}  \cL_{c_*} : (W \cap X)\to W \cap Y 
\end{equation}
is an isomorphism.
Indeed, in view of \eqref{cLVj}, \eqref{Lc0inv}, \eqref{cLincl},
\eqref{VjMCzj} and recalling 
\eqref{Def W} there exists a formal inverse $\cA : 
\la v_0^{(2)} \ra \oplus\mathscr{W} \to 
\la v_0^{(1)} \ra\oplus\mathscr{W} 
$ of
$\Pi_{W\cap Y}  \cL_{c_*} $,  
where
$ v_0^{(1)}, v_0^{(2)}$ are  
in \eqref{vv},  
defined by 
$$\cA v_0^{(2)} :=-\frac{1}{g}
v_0^{(1)} \, ,\quad
\begin{aligned}
    \cA :=\cM \cC 
    \left(
    \begin{matrix}
      (c_*\partial_x-\im \Omega(D))^{-1} &   0\\
        0   & (c_*\partial_x+ \im \overline{\Omega(D)})^{-1}
    \end{matrix}
    \right)
    \cC^{-1}\cM^{-1}
    \quad\text{on}\   \mathscr{W}\, .
\end{aligned}
$$
By \eqref{M} and \eqref{defC} the operator $\cA$ on $\mathscr{W}$ is given by the matrix of
Fourier multipliers 
\begin{equation}\label{DefAin}
    \left(
    \begin{matrix}
        A_{1,1}(D) & A_{1,2}(D)\\
        A_{2,1}(D) &  A_{2,2}(D)
    \end{matrix}
    \right)
\end{equation}
whose   symbols are (recall that the symbol of $\bar{\Omega(D)}$ is $\Omega_{-\xi}$)
\begin{align}\label{simb1}
    A_{1,1}(\xi)=A_{2,2}(\xi)=
        \frac{1}{2}&
    \frac{2c_*\xi+\Omega_{-\xi}-\Omega_\xi}{(c_*\xi -\Omega_\xi)(c_*\xi +\Omega_{-\xi})}\, ,\\
    A_{1,2}(\xi)=\frac{1}{2} M_\xi^2
            \frac{\Omega_\xi+\Omega_{-\xi}}
        {(c_*\xi-\Omega_\xi)(c_*\xi+\Omega_{-\xi})}&
        \, ,\ 
    A_{2,1}(\xi)=\frac{1}{2} M_\xi^{-2}
            \frac{\Omega_\xi+\Omega_{-\xi}}
        {(c_*\xi-\Omega_\xi)(c_*\xi+\Omega_{-\xi})}\, .\label{simb2}
\end{align}
In view of \eqref{M} the symbol 
$ M_\xi $
has order $ -\frac{1}{4} $ if $ \kappa > 0 $ and order 
$ \frac{1}{4} $ if $ \kappa = 0 $.  
Estimating the orders of 
the symbols \eqref{simb1}-\eqref{simb2}
using \eqref{Omega} and \eqref{expOmega}, the operator $\cA$ in \eqref{DefAin} is a
matrix of Fourier multipliers as 
$$
\left( \begin{matrix}{}
	OPS^{-2} & 	 OPS^{-2}  \\
        OPS^{-1} &   OPS^{-2}
\end{matrix} \right) \quad \text{for }\, \kappa > 0
\quad\text{and}\quad
\left( \begin{matrix}{}
				OPS^{-1} & 	 OPS^{-1}   \\
                    OPS^{-2} &   OPS^{-1}
\end{matrix} \right) \quad \text{for }\, \kappa=0\, .
$$
Here we denote by $OPS^{m}$ a Fourier multiplier operator with a symbol of order $m$, see e.g. \cite{BFM}[Definition 3.3].
Note that for $ \kappa > 0 $ the symbol 
$ A_{1,1}(\xi) = A_{2,2}(\xi) $  has order $ {-2}$, because, by  the 
cancellation \eqref{expOmega}, the symbol  $ \Omega_\xi - \Omega_{-\xi} $ has order $ 0 $. 
In conclusion in both cases $\cA$ can be extended to an operator from $W\cap Y = 
            (\mathscr{W}\oplus \la v_0^{(2)}\ra)  
            \cap Y$ to $W\cap X = 
            (\la v_0^{(1)}\ra \oplus\mathscr{W}) 
            \cap X$. 
This proves 
that $\di_w G(c_*,0,0)$ in \eqref{dG0} is an isomorphism.

The  existence of a unique 
analytic solution $w(c, v) $ of \eqref{range} defined for $ (c, v) $ close to $ (c_*,0) $
follows by 
the analytic 
implicit function theorem, 
see e.g. \cite{BuTo}[Section 4.5].
Since $ G(c,0,0) = 0 $ 
(cfr. \eqref{Gc00}) by uniqueness we have $w(c,0)=0$ for any $ c \in B_r(c_*)  $. Next we compute the derivative
of $ w (c,v) $ with respect to $ v $.
For any $ c $ close to $ c_* $ 
the differential 
$ \di_w G(c, 0,0 ) $ is invertible as  well as 
$ \di_w G(c_*, 0,0 ) $
and 
$$
\di_v w(c,0) = 
- \di_w G(c, 0,0 )^{-1}
 \Pi_{W\cap Y} 
 {\cL_c}_{|V} = 0 
$$
because $\cL_c V \subset V$  for any $ c $, cfr. \eqref{cLVj}. 

Finally in order to prove
\eqref{invarL} note  that by 
\eqref{sym}, \eqref{sym tau} and
\eqref{IsItau}, 
	\begin{equation*}
		\Pi_{W\cap Y} \cF(c, \tau_\theta u)=\tau_\theta  \Pi_{W\cap Y} \cF(c,u) \, , \quad
		\Pi_{W\cap Y} \cF(c, \mathscr{S}u)=-\mathscr{S} \Pi_{W\cap Y} \cF(c,u)
	\end{equation*}
(use also 
$\mathscr{S}^\top =\mathscr{S} = \mathscr{S}^{-1} $).
 Therefore by uniqueness we deduce
\eqref{invarL}.
\end{proof}	
	
	In view of  the previous lemma the system \eqref{RangeBifurcation} reduces  to solve the bifurcation equation
\begin{equation}\label{bifoL}
    \Pi_V \cF (c,v+w(c,v)) = 0 
\end{equation}
where $ w(c,v)$ is the solution of the range equation. 
 The equation \eqref{bifoL} is still variational since $\Pi_V\cF(c, v+w(c,v))$ is the 
 symplectic gradient  
 of the ``reduced Hamiltonian"
 $ \Phi (c, \cdot) $ defined below.

 \begin{lemma}\label{lemma:bifva}
 {\bf (Variational 
 structure of the bifurcation equation)}
 The function $ \Phi (\cdot , \cdot ) :
B_r(c_*) \times B_r^V(0)  \subset \R \times 
V \to \R $, $ (c, v ) \mapsto \Phi (c,v) $, defined by 
 \begin{equation}\label{ReducedProblem}	
\Phi (c, v ) := \Psi 
(c,v+w(c,v)) = (\mathcal{H}+c\mathcal{I})(v+w(c,v)) \,, 
\end{equation}
 is analytic on $ B_r(c_*) \times  B_r^V(0) $,
 and 	satisfies,  
 for any   $(c,v)\in  B_r(c_*) \times  B_r^V(0) $,
		\begin{equation}\label{symmetriesPhi}
			\Phi (c, \mathscr{S}v)= \Phi (c, v) \, , \quad  
   \Phi( c, \tau_\theta v)= \Phi (c,v) \, , \quad\forall \theta\in \mathbb{R} \, ,  
		\end{equation} 
  and  
\begin{equation}\label{Hamiltonian Vector Field of Phi}
			\di_v \Phi (c,v)[\widehat{v}]= \di_u \Psi (c, v+w(c, v))[\widehat{v}] =
   \cW(\Pi_V\cF(c, v+w(c,v)),\widehat{v})
   \, , \quad \forall \widehat v \in V \, . 
		\end{equation} 
Therefore if $ \bar v \in B_r^V (0)  $ 
is a critical point 
of $ v \mapsto \Phi(c, v) $ then 
$ \bar u := \bar v + w(c, \bar v )$
is  a solution of $ {\cF}(c, u) = 0 $. 
	\end{lemma}

 \begin{proof}
We first note that the range equation  \eqref{range}
has a variational meaning, being  
$ w(c,v ) $ a critical point 
of the  functional
$ w \mapsto \Psi(c, v + w ) $. Indeed
for any  
$ \widehat w \in W  $ we have,
recalling  \eqref{sympl gradient}, 
\eqref{eq:H}, 
    \begin{align}\label{range var}
\di_u \Psi( c, v+w(c,v))[\widehat w]  
   & = 
            \cW(\cF(c,v+w(c,v)),
            \widehat w )\, 
            \notag \\
            & 
            =
          \cW(
          \underbrace{\Pi_W\cF(c, v+w(c,v))}_{=0 \ \text{by} \  \eqref{range}}, \widehat{w}) = 0  \, . 
    \end{align} 
		Then, differentiating \eqref{ReducedProblem} with respect 
  to $v \in V $ in the direction $ \widehat v \in V $, we get 
    \begin{align}
			\di_v \Phi (c,v)[\widehat{v}]
   & = \di_u \Psi (c, v+w(c, v))[\widehat{v}]
			+
   \underbrace{\di_u \Psi( c, v+w(c,v))[\di_v w(c,v)[\widehat{v}]]}_{=0 \ \text{by}
   \ \eqref{range var} \ \text{and} \
   \di_v w(c,v)[\widehat{v}] \in W
   }  \notag \\
   & = \cW(\cF(c,v+w(c,v)),\widehat{v})  
            =
           \cW(\Pi_V\cF(c, v+w(c,v)),\widehat{v}) \label{4.20}
    \end{align}
    proving \eqref{Hamiltonian Vector Field of Phi}.
    In conclusion, by \eqref{4.20},  
    if
     $ \bar v \in B_r^V (0)  $ 
is a critical point 
of $ v \mapsto \Phi(c, v) $, then 
    $ \Pi_V\cF(c, \bar{v}+w(c,\bar{v})) = 0 $
    because
    $ \cW $ is non-degenerate.
	\end{proof}
 
	Summarizing, we have proved  that the original problem \eqref{eq:H} is equivalent, locally near 
 $ (c_*,0) $,  to find critical points  of the functional $ \Phi (c, \cdot) $ for some value of $ c $. 

The search of 
non-trivial critical points of  $ \Phi (c, v) $  constitutes the core of the proof of the main results  in Sections \ref{sec:speed} and \ref{sec:mom}. 
For that we first  
 expand $ \Phi (c, v) $  close to $ v= 0 $.

\begin{lemma}\label{lem:Gcu}
For any $ (c,v) \in B_r(c_*) \times B_r^V(0) $ 
the  function $ \Phi (c, v) $ in \eqref{ReducedProblem} has the form 
\begin{align}
\Phi (c, v)  & = 
(c-c_*) \cI(v) + G_{\geq 3}(c,v)
= 
\Big( 1 - \frac{c}{c_*} \Big) \cH_2 (v) + G_{\geq 3}(c,v) \label{exparedu} \\
& = \frac12 \cW(\cL_c v, v ) + G_{\geq 3}(c, v) \label{exparedu1}
\end{align}
where  $ \cH_2 (v) := \frac{1}{2}\langle \di_u \nabla\mathcal{H}(0)v,  v\rangle $, the speed $ c_* \neq 0 $
is defined in \eqref{def:c*}, 
and 
$ G_{\geq 3}(c,v) $ is an analytic function
vanishing  at $ v = 0 $ with cubic order for any $ c \in B_r (c_*) $. 
\end{lemma}

\begin{proof}
A Taylor expansion of 
\eqref{ReducedProblem}, 
using $ \cH (0) =  \cI (0) = 0 $, 
$ \nabla \cH (0) = \nabla \cI (0) = 0  $ 
and \eqref{dw0} gives
\begin{align}
\Phi (c, v) & = 
\frac12 \big\la \di_u \nabla (\cH+c\cI)(0)v,v \big\ra + G_{\geq 3}(c,v) \label{tainter} \\
& \stackrel{\eqref{W}, \eqref{eq:H}}=
\frac12 \cW\big( \underbrace{\di_u \cF(c,0)}_{= \cL_c \, \text{by} \, \eqref{dFc0}}v,v \big) + G_{\geq 3}(c,v) 
\notag \, . 
\end{align}
This proves \eqref{exparedu1}. 
Formula  \eqref{exparedu} follows by \eqref{tainter}  because, 
 by 
\eqref{primoV} and \eqref{MomentumInCoord},  
$$ 
\cH_2(v) = \frac12 \la \di_u \nabla \cH(0)v,v \ra  = - c_* \cI (v ) \, , \quad 
\cI (v) = 
\frac12 \la \di_u \nabla \cI(0)v,v \ra 
$$
for any $ v \in V $. 
\end{proof}

Recalling \eqref{defV} and
\eqref{deco-real} 
any  function of  
$ V = \ker(\cL_{c_*}) $ 
in \eqref{kernelL} can be written as
\begin{equation}\label{vkernel}
v = \sum_{ j \in {\cal V} } 
   \alpha_j v_j^{(1)} + \beta_j v_j^{(2)}    \quad \text{where} \ 
   {\cal V} \mbox{ is  defined in } 
   \eqref{kernelL} \, ,
\end{equation}
 the  vectors $ (v_j^{(1)},  v_j^{(2)})_{j \in {\cal V}} $  
defined in \eqref{vv} 
form a symplectic basis of $ V $, 
and the coordinates 
$ \alpha_{j} := \alpha_{j}(v) $, 
$ \beta_{j} 
:= \beta_{j}(v) $ are 
 given  in  \eqref{coordinatev}. 
In view of 
\eqref{HamiltonianInCoord},
\eqref{MomentumInCoord}, 
 we have 
\begin{equation}\label{HIV}
	\mathcal{I}(v)  = 
			- \frac12 \sum_{
   j \in \cV } j (\alpha_j^2+\beta_j^2) \, , 
   \quad 
            \cH_2 (v) = \frac12 
            \sum_{j \in \cV  }
            \Omega_j (\alpha_j^2+\beta_j^2) \, . 
    \end{equation}

    \begin{remark}
    The cubic Hamiltonian $ G_{\geq 3}(c,v) $ in  \eqref{exparedu}
is in ``Birkhoff resonant normal form", namely  
$ \{ G_{\geq 3}(c,\cdot ), \cI \} = 0 $ where 
$ \{ F,G \} :=  
\cW (X_F,X_G) $ 
is the Poisson bracket between two functions on $ V $.
Indeed the reduced Hamiltonian
$ \Phi(c, \cdot )  $ in 
\eqref{ReducedProblem}
defined on the symplectic space 
$ V $ in \eqref{defV}  
has the prime integral 
$ \cI (v) $,  namely 
$ \{ \Phi(c, \cdot), \cI \} = 0 $. 
\end{remark}
\noindent 
{\bf Notation.} In the sequel 
we  denote  
$ \Phi (c, v) $ equivalently as 
$ \Phi (c, (\alpha_j,\beta_j)_{
j \in {\cal V}}) $. 
\\[1mm]
      By Proposition \ref{KER24}
 the space $ V $ can be either $2$
or $ 4 $ dimensional.

\paragraph{Non resonant case.}\label{parNR} 
If  
$ 
\ker(\cL_{c_*}) = V = \{ v = \alpha_* v_{j_*}^{(1)} +
\beta_* v_{j_*}^{(2)} \} $  is $2 $-dimensional (for simplicity we denote
$ \alpha_* = \alpha_{j_*}$
and $ \beta_* = \beta_{j_*}$), 
the symmetries \eqref{symmetriesPhi} show,
recalling 
\eqref{SRtheta}, that 
$$
\Phi (c, \alpha_*, \beta_* ) = \Phi (c,  
R(-j_* \theta)(\alpha_*, \beta_* )) \, , \ 
\forall \theta \in \R \, , 
\quad \Phi (c, \alpha_*, \beta_* ) = \Phi (c,  \alpha_*, - \beta_* ) \, , 
$$
and therefore the 
functional  
$\Phi (c, \alpha_*, \beta_* ) $
is a radial
(i.e. is a function  of $ \alpha_*^2 + \beta_*^2 $) 
for any $c $. 
This shows that all the 
critical points of $ \Phi (c, \alpha_*, \beta_* ) $
 are 
obtained by rotations of  
critical points of the function 
$ \alpha_* \mapsto 
 \Phi (c, \alpha_*, 0 ) $
of one variable only. 
In view of \eqref{exparedu} and  \eqref{HIV}, we have  
\begin{equation}\label{exparedur}
\Phi (c, \alpha_*,0)  =
- \tfrac12 (c-c_*) j_* \alpha_*^2 + G_{\geq 3}(c,\alpha_*,0) \, .
\end{equation}
Since 
\begin{equation}\label{CR.t}
\pa_c \pa_{\alpha_*}^2 \Phi (c, \alpha_*,0) = - j_* \neq 0  \, , 
\end{equation}
by the  implicit function theorem,
for any $ \alpha_* $ small enough, 
there exists a  unique speed $ c(\alpha_*)$, analytic in $ \alpha_* $, such that $(\pa_{\alpha_*} \Phi)(c(\alpha_*), \alpha_*) = 0$.
Actually \eqref{CR.t}  is nothing but 
the Crandall-Rabinowitz transversality condition,  
which   requires $\pa_c \cL_c v_{j_*}^{(1)}\vert_{c_*}$ to not belong to the range of $$
\cL_{c_*} : \big(\la v_0^{(1)}\ra \bigoplus\limits_{j \in \Z\setminus \{0\}} \la v_j^{(1)}\ra \big) \cap X \to \big(\la v_0^{(2)}\ra \bigoplus\limits_{j \in \Z\setminus \{0\}}\la v_j^{(2)}\ra\big)\cap Y=:Y^{(2)} \, . 
$$
Here  the domain and target spaces have been restricted so that  
$\ker (\cL_{c_*})$ is 1-dimensional and  the range $\cR:=\big(\la v_0^{(2)}\ra \bigoplus\limits_{j \neq j_*}\la v_j^{(2)}\ra\big)\cap Y$ is of codimension 1.
Now, 
in view of  \eqref{exparedu1},
$$
 \pa_c \pa_{\alpha_*}^2 \Phi (c, \alpha_*,0) = \cW\big( \pa_c \cL_c v_{j_*}^{(1)}, v_{j_*}^{(1)} \big)  \, , 
$$
which is non-zero iff  $\pa_c \cL_c v_{j_*}^{(1)} \not\in \langle v_{j_*}^{(1)}\rangle^{\bot_{\cW}} \cap Y^{(2)} \equiv \cR$, namely the transversality condition.

All the
non-trivial solutions of \eqref{eq:H} close to $(c_*,0)$ can be parametrized
as rotations of  the Stokes waves  
\begin{equation}\label{defuep}
u_\e = \e v_{j_*}^{(1)} + 
\underbrace{w 
\big( c_\e, \e v_{j_*}^{(1)} \big)}_{= \cO (\e^2)} \in X 
\, , \quad 
c_\e = c_* + \underbrace{\wt c_\e}_{= \cO (\e^2)}  \, ,  \quad  
v_{j_*}^{(1)} := \vect{M_{j_*} \cos(j_*x)}{M_{j_*}^{-1}\sin(j_* x)}  \, .
\end{equation}
Since the ``amplitude-speed" map 
$ \e \to c_\e $ is real analytic (as proved below \eqref{CR.t}) and  
$ (c, v) \mapsto w(c,v)$
is analytic as a map 
$ B_r(c_*)\times B_r^V\to X$  by Lemma \ref{range equation}, 
the Stokes wave $ \e \to u_\e $ in \eqref{defuep} is analytic
as a map $  \{ |\e| \leq \e_0 \}\to X $. We remind that $ X$  is the space of analytic 
$ 2\pi $-periodic functions in \eqref{spaceX}.
The function 
$ w 
\big( c_\e, \e v_{j_*}^{(1)}) $
has the first component even in $ x $
and the second one odd in $ x $, because
$w(c_\epsilon,\epsilon v_{j_*}^{(1)})=\mathscr{S}w(c_\epsilon, \epsilon v_{j_*}^{(1)})$ by
\eqref{invarL} and  $\mathscr{S}v_{j_*}^{(1)}=v_{j_*}^{(1)}$.

\smallskip
In other words, in the non-resonant case when 
dim $ \ker(\cL_{c_*}) = 2 $,  
any Stokes wave is a 
rotation of  a Stokes wave with
$ \eta (x) $ even and $ \zeta (x) $ odd and 
the applicability of the 
Crandall-Rabinowitz bifucation theorem
is an automatic consequence of the 
Hamiltonian  variational 
structure 
of the equations.

\section{Resonant case}\label{sec:Reso}

We now consider the resonant case
when 
dim $ \ker (\cL_{c_*}) = 4 $.
In view of \eqref{kernelL}, \eqref{defVj} we have 
$$
\ker(\cL_{c_*}) = V = \Big\{ v=  \alpha_{j_*} v_{j_*}^{(1)} +
\beta_{j_*} v_{j_*}^{(2)} +
\alpha_{j} v_{j}^{(1)} +
\beta_{j} v_{j}^{(2)} \colon \alpha_{j_*}, \alpha_{j}, \beta_{j*}, \beta_{j} \in \R\Big\} 
$$
where  
$ j \neq j_* $ is the other integer such that
$
\frac{\Omega_j}{j} =  c_* =
\frac{\Omega_{j_*}}{j_*} $
(cfr. \eqref{omegainter}).

Note  that, by \eqref{symmetriesPhi} and 
\eqref{SRtheta}, 
the function $ \Phi (c, v) $ in \eqref{ReducedProblem} satisfies the symmetries
\begin{align}
& \Phi (c, \alpha_{j_*}, \beta_{j_*}, \alpha_j, \beta_j 
) = \Phi 
\big( c,  
R(-j_* \theta)(\alpha_{j_*}, \beta_{j_*} ),
R(-j \theta)(\alpha_{j}, \beta_{j}) \big) \, , \ 
\forall \theta \in \R \, ,
\label{revs1} \\
& 
 \Phi (c, \alpha_{j_*}, \beta_{j_*},
\alpha_{j}, \beta_{j}) = \Phi (c, 
\alpha_{j_*}, - \beta_{j_*},
\alpha_{j}, -
\beta_{j} ) \, . \label{revs2}
\end{align}
The 
reversibility symmetry 
\eqref{revs2}  implies 
that the derivatives 
$ (\pa_{\beta_j} 
\Phi) (c, \alpha_{j_*}, 0,
\alpha_{j}, 0) = 0 $ $ = 
(\pa_{\beta_{j_*}} 
\Phi) (c, 
\alpha_{j_*}, 0,
\alpha_{j},0 ) $ and thus 
if $(\underline{\alpha}_{j_*}, \underline{\alpha}_{j})$ is a 
critical point of 
\begin{equation}\label{crieven}
(\alpha_{j_*}, \alpha_{j})
\mapsto \Phi (c, \alpha_{j_*}, 0, \alpha_{j}, 0) 
\end{equation}
then   $(\underline{\alpha}_{j_*}, 0, \underline{\alpha}_{j},0)$ is
a critical point of $
(\alpha_{j_*}, \beta_{j_*}, \alpha_j, \beta_j) \mapsto 
\Phi (c, \alpha_{j_*}, \beta_{j_*}, \alpha_j, \beta_j) $. 
The corresponding 
 Stokes wave 
$$ 
u = \alpha_{j_*} v_{j_*}^{(1)} +
\alpha_{j} v_{j}^{(1)} +w(c, \alpha_{j_*} v_{j_*}^{(1)} +
\alpha_{j} v_{j}^{(1)} )
$$ 
has the $ \eta $ component which is {\it even}. The  
 Stokes waves in the orbit
$ \{ \tau_\theta u \}_{\theta \in \R} $  are called 
{\it symmetric}, cfr. \cite{MS}, \cite{Seth}. 
However
\eqref{revs1} does not imply that
all the Stokes waves of \eqref{EqWWZakharovHamNewCoordinates} are symmetric, as shown in \cite{MS}, \cite{Seth}.
 Equivalently 
there could be 
critical points of $ v \mapsto \Phi(c, v) $ which are not obtained by a $ \theta $-translation
of critical points of 
\eqref{crieven}.

\begin{definition}
{\bf (Geometrically distinct critical points) }\label{def:geocr}
Two non-trivial critical points 
of $ \Phi(c, v) $  are geometrically 
distinct  if they are not obtained 
by applying the translation operator $ \tau_\theta $ or the reflection operator $ \mathscr{S} $ to the other one. Equivalently 
if they are not in the same orbit generated by the 
action of $ {\mathbb S}^1 \rtimes \mathbb{Z}_2 \cong O(2) $, cfr. Remark  
\ref{O2action}. 
\end{definition}

We are going to prove
the existence of  non trivial critical $ O(2)$-orbits 
of $ \Phi (c, v) $,  
parametrized by the speed $ c \sim c_* $
in Section \ref{sec:speed}, or the momentum $ \cI (v) = a \sim 0 $, in Section \ref{sec:mom}.

\subsection{Stokes waves parametrized by the speed}
\label{sec:speed}

In this section we prove Theorem \ref{simpl}.
For definiteness in the sequel we  assume that  $ j_*, j < 0 $ 
(recall that  $ j, j_* $ 
have the {\it same} sign by \eqref{segni}).  The other case follows similarly.  
With this choice the momentum  in  \eqref{HIV} is the  positive definite 
quadratic form
\begin{equation}\label{momeI6}
\cI (v) =  \underbrace{\frac12 |j_*| (\alpha_{j_*}^2+\beta_{j_*}^2) + 
\frac12 |j| (\alpha_{j}^2+\beta_{j}^2)}_{=: \| v \|_*^2 }   \ ,  
\end{equation}
 and 
$ \| v \|_* := \cI (v)^{1/2}  $ 
is a norm on $ V $. Thus
the functional $ \Phi(c, v) $ has, by  \eqref{exparedu}, 
 a local minimum at $v =0$ for any $c > c_*$, and a local maximum for any $c < c_*$.

For simplicity of notation  we denote $ B_r^V (0) \equiv B_r^V $.

\smallskip

Consider the analytic function 
$ \Phi(c_*, v) = G_{\geq 3} (c_*, v) $ which vanishes cubically at $ v = 0 $ by Lemma \ref{lem:Gcu}. 
If $ v= 0 $ is {\it not} an isolated critical point of 
$ \Phi( c_* , \cdot  ) $ then alternative $(i)$ of Theorem  \ref{simpl} holds: there exists a sequence $ v_n \to 0 $ of critical points of $ \Phi (c_*, \cdot )$ and thus,  
in view of Lemma \ref{lemma:bifva}, a sequence of solutions
$$
u_n = v_n + w(c_*, v_n) \quad \text{of} \quad   \cF(c_*, u_n) = 0 \, , \quad
\text{with} 
\quad 
v_n \to 0 \, . 
$$
Thus, in the following, we assume
$ v = 0 $ is an isolated critical point of $ \Phi( c_*, \cdot ) $.
Consequently $ v = 0 $ is either 
\begin{itemize}
\item[($ii$)] a strict local maximum or minimum for  $ \Phi( c_*, \cdot ) $;
\item[($iii$)] $ \Phi( c_*, \cdot  ) $ takes on both positive and negative
values near $ v= 0 $;
\end{itemize} 
which correspond to alternatives $(ii)$ and 
 ($iii$) of Theorem \ref{simpl}. 

\medskip

{\it Case (ii)}: Suppose $ v= 0$ is a strict local maximum of $ \Phi( c_*, \cdot  ) $ 
(to handle the case of a strict local minimum just 
replace $\Phi  $ with $ - \Phi $). 
Since  $ \Phi( c_*, 0 ) = 0  $, for $ r > 0 $ small enough, 
\begin{equation}
\label{betanega}
\exists  \beta > 0  \quad {\rm such \ that} \quad 
\Phi_{| \partial B_r^V }( c_* , \cdot ) \leq - 2 \beta \, . 
\end{equation}
By continuity, for $ c $ sufficiently close to  $ c_* $, 
\begin{equation}\label{bouneg}
\Phi_{| \partial B_r^V }( c , \cdot ) \leq - \beta \, .
\end{equation}
By  \eqref{exparedu} and \eqref{momeI6}, for any 
 $ c > c_* $, 
the functional $\Phi(c , \cdot )$ has a local minimum
at $ v =0 $, and there exist $ \rho \in (0,r) $ and  $ \alpha ( c ) > 0 $ such that 
\begin{equation}\label{mpca}
\Phi(c,v ) \geq   \alpha (c) > 0 \, , \qquad \forall 
v \in \partial B_\rho^V \, . 
\end{equation} 
The maximum 
$$
{\overline m}(c) := 
\max_{v \in \bar{B_r^V}} \Phi (c, v ) \geq \alpha (c) > 0  
$$
is attained at a point $ \bar v $
in $ B_r^V $ 
because $ \Phi( c , \cdot )$ 
is negative on $ \partial B_r^V $
by \eqref{bouneg}. Furthermore, $ \bar v \neq 0 $ 
because $ \Phi (c, \bar v )  = {\overline m}(c) >0 $ and 
$ \Phi( c , 0 ) = 0 $. 
To find another
geometrically distinct (cfr. Def. \ref{def:geocr}) non trivial  
critical point of $\Phi (c , \cdot  ) $  we define
the Mountain Pass critical level 
\begin{equation}\label{mplevel}
{\underline m}(c) := \inf_{\gamma \in \Gamma} \max_{t \in [0,1]} \Phi(c, \gamma(t))
\end{equation}
where the minimax class $\Gamma $ is 
\begin{equation}
\label{casoa}
\Gamma := \Big\{ \gamma \in C([0,1], \bar{B_r^V}) \ : \  \gamma(0) = 0 \ {\rm and}  \  \gamma (1) \in 
\partial B_r^V  \Big\} \, .
\end{equation}
Since any path $ \gamma \in \Gamma $ intersects $ \{ v \in V \, : \, \|v\| = \rho \} $, 
by \eqref{mpca},  
\begin{equation}\label{cposi}
{\underline m}(c) \geq \alpha (c) > 0 \, . 
\end{equation}
To prove that  ${\underline m}(c) $ is a critical value, 
we can not
directly apply 
the Mountain Pass Theorem of Ambrosetti-Rabinowitz \cite{AR} because $\Phi(c, \cdot )$ is 
defined only in a neighborhood of $ 0 $. However, since 
\begin{equation}\label{MPgeo}
 {\underline m}(c ) \geq \alpha(c) >  0 >
 - \beta \geq \Phi(c, \cdot  )_{|\partial B_r^V} \, , 
\end{equation} 
we adapt its proof showing that $ {\underline m}( c ) $ is a critical value. 
The following lemma holds. 

\begin{lemma}\label{lemPS}
There exists a Palais-Smale sequence 
$ \{v_n\}_{n \geq 0} \subset B_r^V $ at the level $  {\underline m}(c ) $, i.e.  such that  
\begin{equation}
\label{mp1}
\Phi(c,  v_n  ) \to  {\underline m}(c) \, , \quad  \nabla_v \Phi(c,  v_n ) \to 0 \, \mbox{ as } n \to \infty \ . 
\end{equation}
\end{lemma}

The proof is based on a classical deformation argument, that we report in detail  in Appendix \ref{sec:App} for  completeness. 

\begin{remark}\label{rem53}
Let us  comment the idea of the proof for readers not acquainted with critical point theory. 
If, by contradiction, there is no Palais-Smale sequence satisfying 
\eqref{mp1}, 
it is possible to 
deform continuously 
a sublevel  $ \{ \Phi(c,v) \leq \underline{m}(c) + \mu \}$ into  
 $ \{ \Phi(c,v) \leq \underline{m}(c) -  \mu \}$ for  any $ \mu $
 small enough, contradicting 
the min-max definition \eqref{mplevel} 
of mountain pass level 
${\underline m}(c) $. 
We construct such deformations
as  the flow $ \eta^t $ of the negative gradient  $ - \nabla \Phi(c, \cdot) $ multiplied by a cut-off function which is $0$ in a neighborhood of
$ \partial B_r^V $ --where the functional $ \Phi(c, \cdot) $ is negative by \eqref{MPgeo}-- and takes the value $1$ on positive level sets of $\Phi(c,\cdot)$.
In this way the flow $\eta^t:\overline{B_r^V}\to \overline{B_r^V}$ deforms positive sublevels and 
we conclude 
the existence of a Palais-Smale sequence at the
mountain pass level 
${\underline{m}}(c) >0 $, cfr. 
Appendix \ref{def:A1}.
\end{remark}

As a corollary of Lemma \ref{lemPS}, there exists a non trivial critical point $ \underline  v \in B_r^V $ of $\Phi (c,  \cdot )$ at the  level  
 ${\underline m} ( c )$.
Indeed, 
by compactness, up to subsequence,  
$ v_n $ converges to some 
$ \underline  v $ belonging to $ \bar{B_r^V} $. Actually $ \underline  v $ belongs to $ B_r^V  \setminus \{ 0 \}$ 
because, by \eqref{mp1},  
$ \Phi (c, \underline  v ) =  {\underline m}(c )  > 0  $ and  $ \Phi(c, \cdot  )_{|\partial B_r^V}  <  0 $ (by \eqref{bouneg}) and 
 $  \Phi (c, 0 ) = 0 $.

If ${\underline m}(c) < {\overline m}(c) $ then $ \Phi(c, \cdot )$ has two 
geometrically distinct critical points   
(since, by \eqref{symmetriesPhi}, geometrically non 
distinct critical points have the same value  of $ \Phi (c, \cdot) $). 
If $ {\underline m}(c ) = {\overline m}(c) $
then ${\overline m}(c)$ equals the maximum of 
$ \Phi (c, \cdot )$ over {\it every} curve in
$ \Gamma $. Therefore there is a maximum of $ \Phi (c, \cdot )$  on each curve $ \gamma \in \Gamma $ 
defined in \eqref{casoa} and then there are  
infinitely many  geometrically distinct critical points of $ \Phi (c, \cdot )$.
Notice that any $ O(2) = \mathbb{S}^1  \rtimes \Z_2 $ orbit 
is $1$-dimensional and therefore cannot separate the four dimensional domain 
$ B_r^V $ of $ \Phi(c,v) $. 
In any case alternative ($ii$)
holds. 
\\[2mm]
\indent
{\it Case (iii).} 
In this case, since $ \Phi( c_*, \cdot  ) $ 
takes on both positive and 
negative values near $ v = 0$, the functional
$ \Phi( c, \cdot ) $ possesses the Mountain-Pass geometry both
for $ c > c_* $ and $   c < c_* $.
    The 
    min-max argument is more subtle than in the previous case
because \eqref{betanega}
and \eqref{bouneg}
do not hold anymore.
In this case $\partial B_r^V$ intersects the positive levels of $\Phi (c, \cdot) $ and thus the construction outlined in Remark \ref{rem53} to obtain a Palais-Smale sequence  fails.

\smallskip

The key argument is to use 
a `stable Conley 
isolating block" $ W \subset B_r^V $
for the 
isolated degenerate critical point $ v = 0 $ of $ \Phi(c_*, v ) $, that we  construct 
in  Proposition \ref{GromollMeyer} 
(here it is also called a Gromoll-Meyer set).  Associated to 
$ W$ we consider its 
``exit set"
\begin{equation}\label{W-exit}
W_- := \big\{ v \in W \ : \ \eta^t (v) \notin W \ \forall t > 0 \ \text{near} \ 0 \}   
\end{equation}
where
$ \eta^t (v) $ is the negative gradient flow generated by $ - \nabla_v  \Phi( c_*, v ) $. 
The exit set $W_- $ lies in a negative 
level set of $ \Phi (c_*, \cdot ) $, i.e. $ \Phi (c_*, \cdot )_{|W-} < 0  $ 
and it is not empty 
since $\Phi( c_*, v ) $ assumes negative values  arbitrarily close to  $ v = 0 $, Proposition \ref{GromollMeyer}-({\bf II}).

Define the Mountain Pass level
\begin{equation}\label{mountain pass case iii}
m ( c) := \inf_{\gamma \in \Gamma} 
\max_{t \in [0,1]} \Phi(c, \gamma(t)) 
\end{equation}
where 
\begin{equation}\label{casult}
\Gamma := \Big\{ \gamma \in C([0,1],  W)  \  :
\ \gamma ( 0 )  =  0 \  {\rm and}  \ \gamma(1) \in 
W_-  \Big\}  \, .
\end{equation}
Note that by Proposition \ref{GromollMeyer}-Item {\bf (II)} there is a continuous path joining
$0$ and $ W_- $ and so the min-max class  
$ \Gamma $ is not empty. 

For $ c > c_* $
we deduce that  $ m(c ) \geq  \alpha(c ) > 0 $, arguing as 
for \eqref{cposi}.

\begin{lemma}\label{lemPS1}
There is  a
Palais-Smale sequence 
$ \{v_n\}_{n \geq 0} \subset W $ 
inside the stable Conley 
isolating block $W$ of $ v = 0 $
at the level 
$ m(c ) > 0 $, i.e.  
such that 
\begin{equation}
\label{psiii}
\Phi(c,  v_n  ) \to  m(c ) 
\, , \quad  \nabla_v \Phi (c,  v_n ) \to 0 \, \mbox{ as } n \to \infty  \ . 
\end{equation}
\end{lemma} 


The existence of a Palais-Smale sequence as in \eqref{psiii}
follows 
by a deformation argument, detailed at the end of Appendix \ref{def:A1}.

\begin{remark}
Let us explain why the properties of $ W $  allow to implement min-max arguments. 
    Since the exit set $W_-  $ 
    in \eqref{W-exit} of the Conley 
isolating block $ W$ is  contained in a negative level set by Proposition \ref{GromollMeyer}-({\bf II}), we construct  the deformations $ \eta^t $ as the flow of $-\nabla \Phi(c,\cdot)$ multiplied by a cut-off function which is $ 0 $ on $W_-$ and is equal to $1$ on positive level sets.
    We  conclude the existence of a Palais-Smale sequence at 
    the mountain pass level $m (c) > 0 $.
\end{remark}

By compactness $ v_n $ converges to a critical 
point $ \bar v  \in W \subset B_r $ of $ \Phi (c ,  v) $. 
Again  $ \bar v \neq 0 $ 
because $\Phi (c ,  \bar v ) = m (c) > 0  $ and 
$ \Phi (c,  0) = 0 $. 

For $ c < c_* $
we repeat the previous argument for  
$ - \Phi(c, v) $
(note that 
$ - \Phi(c_*, v) $ assumes negative values arbitrarily close to $ v = 0 $). 
The proof of Theorem \ref{simpl}  is complete.

\begin{remark}
    The number of geometrically distinct critical points proved in Theorem \ref{simpl} coincides with the number provided in \cite{FR} (which is expected to be the optimal one) in the present case in which  dim $ V = 4 = 2 n $.
\end{remark}

\subsection{Stokes waves parametrized by the momentum}
\label{sec:mom}

In this section we prove  Theorem \ref{simplCN}. 
For definiteness in the sequel we  assume that  $ j_*, j < 0 $ 
(recall that  $ j, j_* $ 
have the {\it same} sign by \eqref{segni}) and 
$ a > 0 $. 
The other case follows similarly.

We define the functional
    \begin{equation}\label{Psia}
        \Psi_a(c,u):= \cH(u)+c(\cI(u)-a)
    \end{equation}
which differs from $ \Psi (c,u) $ 
defined in \eqref{Variational} just by a constant. 
Thus $ \di_u \Psi_a (c,u) = \di_u \Psi (c,u)  $ and  
therefore a
critical point  of $ u \mapsto \Psi_a(c,u) $ is  a 
solution of \eqref{eq:H}.

    Now we do not fix the speed $ c $,  as in the previous section,   
    but we look for $c(v)$ such that  
     \begin{equation}\label{eqdef}
            (\di_v \Phi) (c(v), v)[v] =0\, , 
        \end{equation}
        for any $v\neq 0$ sufficiently small, 
        namely such that the radial derivative of the function 
        $ v \mapsto \Phi(c,v) $ 
        defined in \eqref{ReducedProblem} vanishes. 
        We mention that the choice for $ c(v) $  in \cite{CN} is different. 
    \begin{lemma}\label{c}
        For any $ v \in B_{r}^V \setminus\{0\} $ 
        (with a possibly smaller $ r > 0 $)
        there  exists
        a unique $ c(v) \in \R  $ solving \eqref{eqdef}
        and satisfying the following properties:   
        the function   $ v \mapsto c (v) $ is  analytic in $B_{r}^V \setminus\{0\}$,
        \begin{equation}
        \label{cc*}
            c(v) = c_* + \cO (\|v\|) \, , 
            \quad 
            \di_v c(v) =   \cO(1) \quad \text{as}  \ v \to 0 \, , 
            \end{equation}
            and
            $ c(\mathscr{S}v)=c(v) $, 
            $  c(\tau_\theta v)=c(v) $ for any $ \theta \in \R $.     
    \end{lemma}
    
    \begin{proof}
 By Lemma   \ref{lem:Gcu} we have that 
 \begin{equation}\label{difphi}
(\di_v\Phi)(c, v) [v]  =
        (c-c_*) \di_u \cI (v)[v] + \di_v G_{\geq 3}(c,v)[v] =
        2 (c-c_*) \underbrace{\cI (v) }_{= \| v \|_*^2 } +  
        {\mathtt G}_{\geq 3}(c,v)
\end{equation}
where 
$ {\mathtt G}_{\geq 3}(c,v) $ is  an 
analytic function 
in $ B_r(c_*) \times B_r^V $
satisfying
\begin{equation}\label{Gannula}
\di_v^\ell 
{\mathtt G}_{\geq 3}(c,0) = 0 \, , 
\quad \forall \ell = 0,1,2, \, , \quad\forall c \in (c_*-r, c_*+r)   \, . 
\end{equation}
In view of \eqref{difphi}
and setting $ c = c_* + \Delta $, 
the equation \eqref{eqdef} is equivalent, for any $v \neq 0 $,  to look for a fixed point of
$$
\Delta   =  - \frac{{\mathtt G}_{\geq 3}(c_* + \Delta ,v)}{2 \| v \|_*^2} =: F(\Delta , v) \, . 
$$
By \eqref{Gannula} 
there exists
$ K > 0 $ such that for any  
$ \Delta
  \in B_r(0) $,  any $ v \in B_r^V \setminus \{0\} $,   
$$
| F(\Delta,  v)| 
  \leq K \|v\| \, , \quad  
|\pa_\Delta F(\Delta,  v)| = 
  \frac{| \pa_c {\mathtt G}_{\geq 3}(c_* + \Delta,v)|}{2 \| v \|_*^2} 
  \leq K \|v\| \, .
$$
As a consequence for any
$ 0 < \| v \| < r / K < 1$, the map
$ F( \cdot , v )$ is a contraction 
on  $ B_r (0)$.
Hence for any $v\in B_{r/K}^V(0)\setminus\{0\}$ there exists a unique fixed point $\Delta(v)$ of $F(\cdot,v)$, namely a solution $ c(v) := c_* + \Delta (v)  $ of  \eqref{eqdef} (in the smaller domain $ r / K $). 

The function $v\mapsto c(v) = c_* + \Delta (v) $ is analytic by applying the implicit function theorem  to the analytic function 
$ H(\Delta, v) := \Delta  - F(\Delta, v) $ which vanishes 
at $ H(\Delta (v) , v) = 0 $
and satisfies $ \pa_\Delta H(\Delta, v) = 1 - \pa_\Delta F(\Delta, v) \neq 0 $.

The function $\Delta(v) 
:= c(v) - c_*  $ satisfies 
$|\Delta (v) |=|F(\Delta(v),v)|\leq K\|v\|$
proving the first bound in \eqref{cc*}. 
Taking the differential of the equation \eqref{eqdef} which is satisfied identically for $c=c(v)$ and using \eqref{difphi} 
$$
    \di_vc(v)\big( 2\|v\|_*^2+\partial_cG_{\geq3}(c(v),v)\big)+2 \big(c(v)-v\big) \di_u\cI(v)[\widehat{v}]+ \di_v G_{\geq3}(c(v),v)[\widehat{v}] = 0 \, .
$$
Thus, for $v$ sufficiently small, using  $c(v)-c_*=\cO(\|v\|)$
we obtain $\| \di_v c(v)\|=\cO(1)$ as $v\to 0 $.
This proves the second bound in \eqref{cc*}.
The last invariance property  follows by \eqref{symmetriesPhi} and  uniqueness.
    \end{proof}
Next we define the set 
    \begin{equation}\label{sfereSa}
           \cS_a:= 
           \cS_{r,a} := \Big\{ v\in B_r^V \, : \, 
           I(v) 
           :=
           \mathcal{I}(v+
           w(c(v),v))  = a \Big\}  \, . 
     \end{equation}      
Since $ I(v) $ is asymptotic, for $ v \to 0 $,  to 
the homogeneous quadratic 
function $ \cI(v) $ in \eqref{momeI6},  
the set $\mathcal{S}_a$ is, 
for $ r, a > 0  $
small,   
an  
ellipsoid-like compact manifold.   
A supplementary  to the 
tangent space $ T_v \mathcal{S}_a $
is the $1 $ dimensional space 
spanned by $ \langle v \rangle $,  namely 
 $ V = T_v \mathcal{S}_a \oplus \langle v \rangle $. 
    
    \begin{lemma}\label{lem:Sacomp}
        There exist $ r_0, a_0>0$ such that for any $ r\in (0,r_0) $ and 
        $a\in (0,a_0) $, the set   $\mathcal{S}_a$ in \eqref{sfereSa} is a compact manifold contained in  
        $ \cA_{a} := \{ v \in B_r^V \, : \,  (a/2)^{1/2} \leq \| v \|_* \leq (2a)^{1/2} \}  $. 
    \end{lemma}
    
    \begin{proof}
    By \eqref{dw0}, \eqref{cc*},  \eqref{momeI6}, the functional
    $  I (v) $ in \eqref{sfereSa} 
    has the expansion 
\begin{equation}\label{svilI}
I(v)  = 
           \cI (v) + \cO(\| v \|^3)
           = \| v \|_*^2 + \cO(\| v \|^3)
\end{equation} 
for any $ \| v \| \leq r < r_0 $ small enough.  As a consequence there is $ a_0 > 0 $ such that, 
for any $ 0 < a < a_0 $, 
any $ v $
in $ \cS_a $ satisfies 
$ \sqrt{a/2} \leq \| v\|_* \leq \sqrt{2a} $. 
The set $ \cS_a $ is thus 
contained in the annulus 
$ \cA_a $ and
closed (the function $  I (v) $ is continuous), thus compact.  
On the set $ \cA_a $ the 
function $ c(v) $ is actually analytic.  We now prove that 
$ \cS_a $ is a manifold.
Differentiating $ I(v) $ in \eqref{sfereSa} 
at any $ v \in B_r^V $
in the direction $ v $ (radial derivative), we have
$$
\di_v  I(v)[v] = 
\di_u \cI (v + w(c(v),v))
\big[v +
\di_v w(c(v),v)[v] +
\pa_c w(c(v),v) \di_v c(v)[v] 
\big] \, . 
$$
By \eqref{dw0}, \eqref{cc*},  \eqref{momeI6},
there exist constants $ 0 < 
c_1 <  c_2 $ such that
\begin{equation}\label{deradial}
c_1 \| v \|_*^2 \leq 
\di_v  I(v)[v] 
\leq c_2 \| v \|_*^2 \,
\end{equation}
for any  $ \| v \|_* \leq r < r_0  $. As a consequence
the set $ \cS_a $ is a manifold. 
    \end{proof}
    
Finally, for any
$a\in (0,a_0)$ defined in Lemma \ref{lem:Sacomp}, we define 
the functional  
\begin{equation}\label{defPsia}
\phi_a  : B_r^V  \to  \R  \, , \quad 
    \phi_a (v):= \Psi_a \big( c(v), v + \breve w(v) \big) 
\end{equation}
where $ \breve w(v) := w(c(v), v)$. 
Then  $\phi_a (\tau_\theta v) = \phi_a (v) $ for any $\theta $ and 
$\phi_a ( \mathscr{S} v) = \phi_a (v) $. 


    \begin{lemma}\label{lemma59}
        If 
        $ \bar v \in \mathcal{S}_a $ is a critical point of  $\phi_a : \mathcal{S}_a \to \R $ then 
  $ \bar u := \bar v + \breve w(\bar v ) $
        is a  solution of $ {\cal F}(c (\bar v), 
         u ) = 0 $
         with 
momentum $ \cI ( \bar u  ) = a $ 
        and  speed $ c( \bar v)  $. 
    \end{lemma}
    
    \begin{proof}
Differentiating \eqref{defPsia}
at any $ v \in \mathcal{S}_a $, 
in any direction
$ \widehat{v} \in V $, 
 \begin{align}
            \di_v\phi_a (v)[\widehat{v}]
            & = 
            (\pa_c \Psi_a) ( c(v), \breve w(v) ) \, \di_v c(v) [\widehat v ] + \di_u \Psi 
            (c(v), v + \breve w(v))[\widehat{v}]  \notag 
            \\
            & \qquad \qquad \qquad 
            \qquad \qquad \quad \qquad 
            + \underbrace{\di_u \Psi 
            (c(v), v+ \breve w(v))[
            \di_v \breve w(v)[\widehat{v} ] ]}_{=0
            \ \text{by} \ \eqref{range var} \ \text{and} \ \di_v \breve w(v)[\widehat{v}]\in W } 
            \,  \notag \\        
            & 
            \stackrel{\eqref{Psia}} = 
            \underbrace{( {\cal I}( v+ \breve w(v) ) - a)}_{=0
            \ \text{by} \ \eqref{sfereSa} } \, 
            \di_v c(v) [\widehat v ] + \di_u \Psi 
            (c(v), v+ w(c(v),v))[\widehat{v} ] \notag  \\
            & 
             \stackrel{\eqref{Hamiltonian Vector Field of Phi}} = 
             (\di_v \Phi) 
            (c(v), v)[\widehat{v} ] \,. 
        \label{FundamentalEqForWMVariational Argument}
    \end{align} 
        Let $\overline{v} $ 
        be a critical point
        of $\phi_a : \mathcal{S}_a \to \R  $. 
        Then, by \eqref{FundamentalEqForWMVariational Argument} and recalling \eqref{sfereSa}, there exists a Lagrange multiplier $\mu \in \R$ such that
        \begin{equation}\label{LagrangeMultiplier}
            \di_v\Phi (c(\overline{v}), \overline{v})[\widehat{v}]
            = \mu \, \di_v I(\overline{v})[\widehat{v}] \, ,  \quad \forall \widehat{v}\in V \, .
        \end{equation}
        Taking $\widehat{v} = \overline{v} $ inside \eqref{LagrangeMultiplier} we get
        $$ 
        0 \stackrel{\eqref{eqdef}} = \di_v \Phi(c(\overline{v}), \overline{v})[\overline{v}] = \mu  \di_v  I(\overline{v})[\overline{v}] 
        $$
        and, since  
        $    \di_v I(v)[v]\not=0 $
        by \eqref{deradial} 
        and $ \| v \|_*^2 \geq a/ 2 > 0 $, we deduce that $ \mu  = 0 $.
        Therefore, by  \eqref{LagrangeMultiplier}, 
        $ \bar v $ is a critical point of 
        $ v \mapsto  \Phi ( c(\bar v), v )$ and  Lemma 
        \ref{lemma:bifva} implies Lemma \ref{lemma59}.  
    \end{proof}

 Since $ \mathcal{S}_a $ is a compact manifold (Lemma \ref{lem:Sacomp}),
the  
functional $ \phi_a $ on  $\mathcal{S}_a $
possesses at least a minimum $m $ and maximum $M $.
If $ m < M $ 
a minimum point $ \underline{v} $, i.e. 
$ \phi_a (\underline{v}) = m $,  and a maximum point 
$ \bar v  $, i.e. $ \phi_a (\bar v) = M $, are geometrically distinct.
If $ m = M $ the function $ \phi_a $ is constant on $ \mathcal{S}_a $
and then there are  
infinitely many  geometrically distinct critical points of $ \phi_a $ 
(any $ O(2) = \mathbb{S}^1  \rtimes \Z_2 $ orbit 
is $1$-dimensional and 
$ \mathcal{S}_a $ is $3$-dimensional).
In both cases,
in view of Lemma  \ref{lemma59},  Theorem \ref{simplCN}
is proved. 

\begin{remark}
    The number of $2$ geometrically distinct critical points obtained  in Theorem \ref{simplCN} by  
    topological arguments (the maximum and the minimum) is in general optimal.
    The function 
    $$
    f:{\mathbb S}^{3} := \{ (z_1,z_2)\in \C^2 \ : \ 
    |z_1|^2 + |z_2|^2 = 1 \}\to \R \, , \ (z_1,z_2)\mapsto |z_1|^2 \, , 
    $$
    is invariant under the actions of the group $\mathbb{S}^1 \rtimes \Z_2$ where 
    $    {\tau_\theta} (z_1,z_2) := ( e^{-\im j_1 \theta} z_1, e^{-\im j_2 \theta} z_2) $
    and 
    $    \mathscr{S}  (z_1, z_2) := (\bar z_1, \bar z_2) $, cfr. \eqref{SRtheta}. The function
$ f $     has only  two critical orbits 
    $\{(z_1,0) \ : \  |z_1|=1\}$ and $\{(0,z_2) \ : \  |z_2|=1\}$.
\end{remark}


\appendix

\section{Appendix}
\label{sec:App}

In this appendix we prove Lemmata  \ref{lemPS} and \ref{lemPS1}
and the existence of a stable  Conley isolating block. 

\subsection{Existence of Palais-Smale sequences}\label{def:A1} 

The argument is based on deforming 
the sublevels of $ \Phi (c, \cdot) $ 
and 
exploits a topological change 
between  $ \{ \Phi(c,v) \leq \underline{m}(c) + \mu \}$ and  
 $ \{ \Phi(c,v) \leq \underline{m}(c) -  \mu \}$: in view of \eqref{mplevel}
the first set is path connected whereas 
the second one is not.  For simplicity of notation we denote $ \nabla \Phi  = \nabla_v  \Phi  $.  

\paragraph{Proof of Lemma  \ref{lemPS}.}
We claim the following.
\\[1mm]
{\bf Claim:} {\it For any $ 0 < \mu < \underline{m}(c) \slash 2  $ there exists $ v \in B_r^V $ such that}
\begin{equation}\label{cmu}
\underline{m}(c) - \mu \leq \Phi (c, v) \leq \underline{m}(c) + \mu \quad {\it and} \quad \| \nabla \Phi ( c, v ) \| 
< 2 \mu \, .
\end{equation}
Then, choosing $ \mu = 1 \slash n  $ for any $ n $ large enough    we find a Palais-Smale sequence $ v_n $
at the level $ \underline{m}(c) $, i.e. satisfying \eqref{mp1}. 
\\[1mm]
{\bf The Deformation Argument.} There is $ \delta  > 0 $ such that 
the functional $ \Phi(c, v) $
is defined on the open ball 
$ B_{r+\delta}^V $
and it is negative on the annulus 
$ B_{r+\delta}^V \setminus  B_{r-\delta}^V $  by  \eqref{bouneg}.

For any $ 0 < \mu < \underline{m}(c) \slash 2  $,  
we define the sets
\begin{align}
{\tilde N} & := \{ v \in 
B_{r+\delta}^V \, : \, 
|\Phi (c, v) - \underline{m}(c) | \leq \mu  \ \ {\rm and} \ \ \| \nabla \Phi (c, v)\| \geq  2 \mu \} \label{defNtilde} \\
N & := \{ v \in B_{r+\delta}^V  \, : \, 
|\Phi (c, v) - \underline{m}(c) | < 2 \mu \ \ {\rm and} \ \ \| \nabla \Phi (c, v)\| >  \mu \} \, . \label{defN}
\end{align}
Clearly $ {\tilde  N} \subset N $. Note also that 
$ \partial B_{r}^V \subset N^c 
:=  B_{r+\delta}^V \setminus N  $
since, 
for any $ v \in \partial B_r^V  $ we have $ \Phi (c, v) < 0 $
by \eqref{bouneg}, 
$ 0 < 2 \mu < \underline{m}(c) $, 
and then $ v \in N^c $.

The sets  $ {\tilde N} $ and 
$ N^c  $ are closed
and disjoints  and therefore there
is a locally Lipschitz non-negative function 
$ g : B_{r+\delta}^V  \to [0,1]$ such that
\begin{equation}\label{gsiannu}
g = 1 \ {\rm on } \  {\tilde N} \, , \qquad   
g = 0 \  {\rm on } \ N^c  \, , 
\end{equation}
for example
$ g(v) :=  \frac{d(v,N^c)}{d(v,N^c)+ d(v, {\tilde N})} $
where  $d( \cdot , \cdot )$ denotes the distance function in $ V $.

We define the locally Lipschitz 
and bounded vector field
\begin{equation}\label{eq:pse}
X(c, v) := - g(v)\frac{\nabla \Phi(c, v)}{\| \nabla \Phi (c, v)  \|}  
\end{equation}
which is well defined because if 
$\| \nabla \Phi (c, v)  \| < \mu$  then $ v \in N^c $ (cfr. \eqref{defN}) and so 
$ g(v) = 0 $ by \eqref{gsiannu}. Furthermore
 the vector field $ X(c, \cdot )  $ 
vanishes on $ \partial B_r^V $
since $ \partial B_r^V \subset N^c $.

For each $ v \in B_{r+\delta}^V $, the unique solution of the Cauchy problem 
\begin{equation}\label{eq:cau}
 \frac{d}{d t} \eta^t  ( v) = X(c,  \eta^t  (v))  \, , \quad    
\eta^0  (v) = v \, , 
\end{equation}
is defined for any $t \in \R $ (since $ X $ is bounded)
and 
\begin{itemize}
\item[$(i)$] $\eta^t (\cdot )$ is a homeomorphism of $ B_{r+\delta}^V $. 
\end{itemize}
Furthermore, since $ X(c, v) = 0 $ for any  $ v \in N^c $, by \eqref{eq:pse}
, \eqref{gsiannu}, 
\begin{itemize}
\item[$(ii)$] $\eta^t  (v) = v $, for any $  t $, if 
$ | \Phi (c, v) - \underline{m}(c) | \geq 2 \mu$ or if $ \| \nabla \Phi(c,  v ) \| \leq \mu$, in particular 
$ \eta^t (0) = 0 $
and $\eta^t (v ) = v  $ 
for any  
$ v \in
\partial B_r^V \subset N^c $. 
\end{itemize}
The properties ($i$)-($ii$) imply that 
the min-max class 
$ \Gamma  $ defined in \eqref{casoa} is {\it invariant} under 
the flow of $ X (c, \cdot) $, namely
\begin{equation}\label{invarianza}
\text{for any path 
$ \gamma \in \Gamma $,
for any 
$ t \in \R $, 
the deformed path  
$ \eta^t \circ \gamma $ belongs to $ \Gamma $.} 
\end{equation}
Furthermore, by \eqref{eq:pse}, \eqref{eq:cau}, 
\begin{itemize}
\item[$(iii)$] 
for any $ v \in B_{r+\delta}^V $, for any $  t \in \R $
\begin{equation}\label{descent}
\frac{d}{d t} \Phi(c, \eta^t  (v)) = - g(\eta^t (v) )\| 
\nabla \Phi (c, \eta^t  (v)) \| \leq 0 \, . 
\end{equation}
\end{itemize}
We now prove the claim \eqref{cmu}.
Arguing by contradiction suppose there exists $ 0 <  \mu  < \underline{m}(c) \slash 2 $ such that 
\begin{equation}\label{smg}
\{ v \in B_{r}^V \, : \, \underline{m}(c) - \mu \leq \Phi(c,v) \leq \underline{m}(c) + \mu \}
\subseteq  
\{ v \in B_{r}^V \, : \, \| \nabla \Phi (c, v ) \| \geq  2\mu \} \, . 
\end{equation}
By the definition of $ \underline{m}(c) > 0 $ in 
\eqref{mplevel} there exists a path $ \gamma \in \Gamma $ (see \eqref{casoa}) such that
\begin{equation}\label{pathMP}
\max_{t \in [0,1]} \Phi (c, \gamma (t) ) \leq \underline{m}(c) + \mu\, .
\end{equation}
But we claim that
\begin{equation}\label{below}
\Phi( c, \eta^1(\gamma ([0,1])) \leq \underline{m}(c)  - \mu \, , 
\end{equation}
implying the contradiction
$$
\underline{m}(c) := 
\inf_{\gamma \in \Gamma} \max_{t\in [0,1]} \Phi ( c, \gamma (t) ) \leq 
\max_{t\in [0,1]} 
\Phi ( c, 
\underbrace{\eta^1  (\gamma (t) )}_{\in \Gamma \, \text{by} \, \eqref{invarianza}} ) 
\leq \underline{m}(c) - \mu \, . 
$$
Let us prove \eqref{below}. 
Pick a point $ v \in \gamma([0,1]) $. If $ \Phi (c, \eta^t (v)) \leq \underline{m}(c) - \mu$ 
for some $0 \leq t \leq 1 $ thus
$ \Phi (c, \eta^1  (v)) \leq \underline{m}(c) - \mu$   since
$ \Phi ( c, \eta^t  (  v) )$ is not-increasing  by \eqref{descent}. 
If 
$ \Phi (c, \eta^t ( v)) > \underline{m}(c)  - \mu$ 
for any $t \in [0,1]$, then,  
by (\ref{smg}), 
\eqref{pathMP}, \eqref{descent}, 
\begin{equation}\label{scende}
 \underline{m}(c) - \mu\leq \Phi (c, \eta^t ( v)) \leq 
 \underline{m}(c) + \mu\, \quad {\rm and} \quad
 \| \nabla \Phi (c,  \eta^t ( v )) \| \geq  2 \mu\, ,
\end{equation}
i.e.  $\eta^t ( v) \in {\tilde N}$ in \eqref{defNtilde}, for any $ t \in [0,1]$.
By (\ref{descent}), since $ g \equiv 1 $ on $  {\tilde N} $
(cfr. \eqref{gsiannu})
\begin{align*}
\Phi(c, \eta^1 (v))  & =  \Phi(c, v)  - 
\int_0^1 g(\eta^t (v )) \, \| \nabla \Phi (c, \eta^t  (v)) \| d  t \\
& \stackrel{\eqref{pathMP}} \leq  (\underline{m}(c) + \mu ) - \int_0^1 \| \nabla \Phi (c, \eta^t (v) ) \| d  t 
\stackrel{\eqref{scende}} 
\leq
\underline{m}(c) + \mu - 
2 \mu= \underline{m}(c)  - \mu
\end{align*}
proving \eqref{below}. 
This 
contradiction 
concludes the proof of Lemma \ref{lemPS}.

\paragraph{Proof of Lemma  \ref{lemPS1}.}
To apply the same argument of Lemma \ref{lemPS}
the main issue is to prove that the min-max class
$ \Gamma $ defined in \eqref{casult} is invariant under the positive flow 
$ \eta^t( v)$ 
generated by the vector field  $ X(c, v) := - g(v) \nabla \Phi (c, v) $
where $ g (v)  $ is the scalar non-negative function defined as in \eqref{gsiannu}
(clearly defining 
$ N, \tilde N $ in \eqref{defNtilde}, \eqref{defN} as subsets of
$ W$). 
This is true, for $ c $ sufficiently close to $ c_* $ 
by  the stability property 
under $ C^1 $-perturbations stated in 
Proposition \ref{GromollMeyer}-({\bf III}). 
If $v \in W_-  $
then $ \Phi(c_* ,v) < 0 $ and, for $ c $ sufficiently close to $ c_* $, 
also $ \Phi(c,v) < 0 $  and thus 
$X(c, v) \equiv 0 $ (because 
$ g(v) = 0 $ on $N^c$).
Therefore $ \eta^t (v) = v $ for any $ t  $ and 
$ \eta^t(\cdot)_{|W_-} = I $.  
On the other hand, the set $ W \setminus W_-  $ 
is invariant for the positive flow $\eta^t(v)$ generated by $X(c,v)$ and consequently any curve $\gamma(\cdot)$ with values in $W$ is
deformed by  $ \eta^t(\cdot ) $ into a curve with values in $W$, proving the invariance of 
the min-max class $\Gamma$ in \eqref{casult}
 under the deformations $ \eta^t(\cdot ) $. 

 \subsection{A stable  Conley isolating block}
 
 Let $ f : B_r \subset V \to \R $ be a function of class $ C^2 $ defined in 
the ball $ B_r $ of radius $ r > 0 $ 
centered at $0 $ in a finite dimensional Hilbert space $ V $
with an isolated critical point at $ v = 0 $ at the level  $ f(0) = 0 $. 
We follow the construction in \cite{ChG}, \cite{GrM}.  
Let $  \varepsilon > 0 $ and $ \delta \in (0,r) $ such that 
$0 $ is the unique critical value in $ [-\varepsilon, \varepsilon] $
and $ v = 0 $ the  unique critical  
point of $ f $ in $ \overline{B_\delta} $. 
Consider the function
\begin{equation}\label{defg}
  g(v) := \lambda \| v \|^2 +  f(v)  
  \end{equation}
where 
\begin{equation}\label{lambr}
    0 < \lambda < \frac{\beta}{2 \delta } \, , \qquad 
\beta :=
\min_{\tfrac{\delta}{2} \leq 
\| v \| \leq \delta}  \| \nabla f (v) \| > 0 \, .  
\end{equation}
We denote $ g_\mu := \{ v \in B_r \ : g(v) \leq \mu \} $.
Given $ \gamma \in (0, \varepsilon)$, $ \mu > 0 $ we  define the sets
\begin{align}
& W := f^{-1} [-\gamma,\gamma] \cap g_\mu = \big\{ v \in B_r \ : \ |f(v)| \leq \gamma, \ 
g(v)\leq \mu \big\} , \label{Wset} \\
& W_- := f^{-1} (-\gamma) \cap W
= \big\{ v \in B_r \ : \ f(v) = - \gamma, \ 
g(v)\leq \mu \big\}  \label{W-} 
\, .  
\end{align}

\begin{lemma}\label{lem:stine}
If  $ \gamma,  \mu $ satisfy 
\begin{equation}\label{ineq}
 0 < \gamma < \min \Big\{ \varepsilon, \frac{3\delta^2 \lambda}{8} \Big\} \, , \quad 
\frac{\delta^2 \lambda}{4} + \gamma < \mu < \delta^2 \lambda - \gamma \, ,    
\end{equation}
then
\begin{enumerate}
    \item[(i)] 
    $ \overline{B_{\delta/2}} \cap f^{-1} [-\gamma,\gamma] 
\subset W \subset B_\delta \cap f^{-1} [-\varepsilon, \varepsilon ]$;
\item[(ii)] 
$ f^{-1} [-\gamma,\gamma] \cap g^{-1}(\mu) \subset  
B_{\delta} \setminus \overline{B_{\delta/2}} $; 
\item[(iii)] 
$ \langle \nabla g(v), \nabla f(v)\rangle >\beta(\beta-2\delta\lambda)>0 $ for any $ 
v \in \overline{B_{\delta}} \setminus 
B_{\delta/2}  $.
\end{enumerate}
\end{lemma}

\begin{proof}
Properties ($i$)-($ii$) are directly verified by the definitions \eqref{Wset}, \eqref{W-}
using \eqref{ineq} and noting that $ \| v \|^2 = \frac{\mu-f(v)}{\lambda}$ for any $v\in g^{-1}(\mu)$. 
Property ($iii$) follows by \eqref{defg} and \eqref{lambr}.
\end{proof}

The set $W $ is a compact neighborhood of $ v = 0 $.  
The set $ W_-$  is a  closed subset of $ W $.

\begin{proposition}\label{GromollMeyer}
{\bf (Stable  Conley isolating block)}. 
The set $ W $ in \eqref{Wset}  has the following properties:
\begin{itemize}  
\item {\bf (I)} for any $ v \in W $ then either 
\begin{itemize} 
\item[(i)] $ f(v) = - \gamma $, {\rm or }
\item[(ii)] 
$ \eta^t (v) \in W $, for any $ t > 0 $ near $ 0 $.
\end{itemize}
\item  {\bf (II)} 
The set $ W_-$ in \eqref{W-}
is  the ``exit" set
of $ W $ with respect to the negative gradient flow  
$ \eta^t (v) $ generated by  $ - \nabla f(v ) $, namely
\begin{equation}
W_- = \big\{ v \in W \ : \ \eta^t (v) \notin W \ \forall t > 0 \ \text{near} \ 0 \} \, . 
\label{W-alto}
\end{equation}
in the negative level set $ f^{-1}(- \gamma )$.
If $ f $ assumes negative values arbitrarily close to $ v = 0 $ then $ W_- \neq \emptyset $
and there is a point $ v \in W_- $ such that $ \eta^t (v) \to 0 $ as $ t \to - \infty $. 
\item {\bf (III)}
{\bf (Stability under $ C^1 $-perturbations)}
If $ F : B_r \to \R $ 
is sufficiently close in $ C^1 $ norm to $ f $
on $ \pa W \setminus W_- $ then, for any $ v \in W \setminus W_- $ 
the negative gradient flow  $ \eta^t_F $  
generated by $ - \nabla F(v)$ satisfies 
$ \eta^t_F (v) \in W $ for any $ t > 0 $ close to $0 $.
\end{itemize}
\end{proposition}

\begin{proof}
Let us prove item {\bf (I)}. Let $ v \in W $. 
Either $ f(v) = - \gamma $ (case ($i$)) or $ f(v) \in (- \gamma,\gamma] $ case ($ii$). In this case
$ \gamma \geq f(v) \geq f(\eta^t(v)) > - \gamma  $ for any $ t > 0 $ near $0 $. 
Furthermore, if $ g(v) < \mu $ then by continuity $ g(\eta^t(v)) < \mu $ for any $ t > 0 $ near $0 $.
If $ g(v) = \mu $ then by Lemma \ref{lem:stine}-($ii$) we deduce that $ v \in 
\overline{B_{\delta}} \setminus B_{\delta/2} $ and thus  Lemma \ref{lem:stine}-($iii$) implies 
that $ \frac{d}{dt} g (\eta^t (v) )_{|t=0} = - \la\nabla g(v), \nabla f(v)\ra < 0 $. In both cases
 $ g (\eta^t (v)) < \mu $  for any $ t > 0 $ near $0 $, thus $ \eta^t (v) \in W $
 for any $ t > 0 $ near $0 $. 

The equality in \eqref{W-alto}
follows by Item {\bf (I)} 
 and since for any $v \in W_- $ we have 
$ \nabla f(v) \neq 0 $ we deduce \eqref{W-alto}. 
The last statement of Item {\bf (II)} follows because the $\omega $ and $ \alpha $ 
limit sets of the gradient flow is not empty and 
contained in the set of critical points of $ f $,  
jointly with item  {\bf (I)}.  

We now decompose $W $ in \eqref{Wset} in disjoint subsets as
$$
W =    \underbrace{\{ v \in B_r \ : \ |f(v)|< \gamma \, , \ 
|g(v)| < \mu \}}_{=:A} \cup  W_- \cup {\cal W}  
$$
where $W_- $ is defined in \eqref{W-} and 
\begin{equation}\label{calW}
{\cal W} := \underbrace{
\Big\{ v \in B_r   \ : \  - \gamma < 
f(v) \leq \gamma \, , \ g(v) = \mu  \Big\}}_{=: {\cal W}_1} 
\bigcup
\underbrace{\Big\{ v \in B_r   \ : \  f(v) = \gamma \, , \ g(v) < \mu  \Big\}}_{=: {\cal W}_2}  
\, . 
\end{equation} 
Note that the set $ {\cal W}_1 $  in \eqref{calW} is included in
$ 
B_{\delta} \setminus \overline{B_{\delta/2} }
$.
It results that
\begin{equation}\label{linsieme}
  A = \overset{\circ}{W}   \qquad \text{and} \qquad 
W_- \cup {\cal W} = \pa W \, .  
\end{equation}
Indeed $ A \subset \overset{\circ}{W} $ trivially by continuity and 
we claim that any $ v \in  W_- \cup {\cal W} $ belongs to the boundary of $ W $.
If $ v \in W_- $ the flow  $ \eta^t (v) \notin W $ for any $ t > 0 $ arbitrarily small.
If $ v \in {\cal W}_1 $   Lemma \ref{lem:stine}-($ii$)-($iii$) implies 
that $ \frac{d}{dt} g (\eta^t (v) )_{|t=0} = - \la\nabla g(v), \nabla f(v)\ra < 0 $, so
$ g (\eta^t (v)) > \mu $ for any $ t < 0 $ arbitrarily small.
If $ v \in {\cal W}_2 $ then $ \frac{d}{dt}  f (\eta^t (v) )_{|t=0} 
= - \| \nabla f (v) \|^2 < 0 $ so
$ f (\eta^t (v)) > \gamma $ for any $ t < 0 $ arbitrarily small. This implies \eqref{linsieme}. 

Let us finally prove item {\bf (III)}. Using that  $ f $ has no critical points 
on the part of the boundary ${\cal W } $ and 
that $ \langle \nabla g(v), \nabla f(v)\rangle >\beta(\beta-2\delta\lambda)>0 $ for any $ 
v \in {\cal W }_1 \subset \overline{B_{\delta}} \setminus 
B_{\delta/2}  $ by Lemma \ref{lem:stine}, 
we deduce that, for $ F $ sufficiently close to $ f $, 
\begin{equation}\label{strong}
\langle \nabla f(v), \nabla F(v) \rangle > 0 \, ,  \quad  \forall v \in {\cal W } \, , \quad 
\langle \nabla g(v), \nabla F(v) \rangle > 0 \, , \quad \forall v \in {\cal W }_1\, . 
\end{equation}
By \eqref{strong} for any $ v \in {\cal W}_1 \cup {\cal W}_2 $ 
we have  that $ \eta_F^t (v) \in \overset{\circ}{W} $ for any $ t > 0 $ close to $0$. 
Thus the flow $ \eta_F^t (v) $ can exit $ W $ only through $ W_- $ and the proposition 
is proved.
\end{proof}

\begin{figure}[h]\label{fig:dis1}
    \begin{center}
    \includegraphics[ width =0.25\textwidth ]{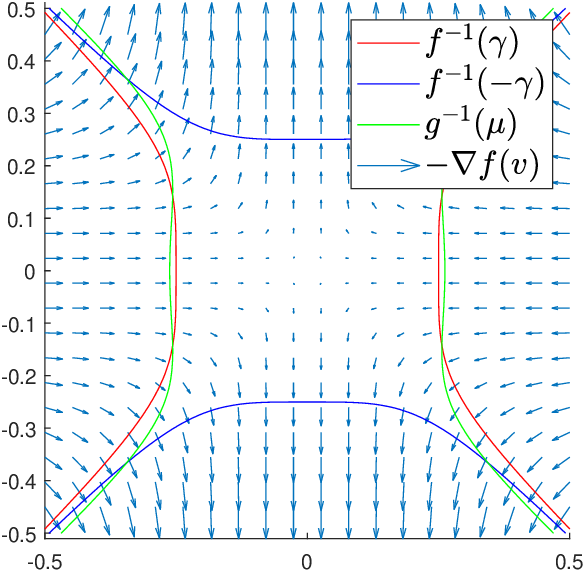}
    \end{center}
    \caption{A stable Conley isolating block for the function $f(x,y)=x^4-y^4$.
    For any function $ F $ which is $ C^1$-close to $ f $, 
    the negative gradient vector field $ - \nabla F $   points inward $ W $ on
     $ \pa W \setminus W_- $ and points outside $W$ on $ W_- $.}
    \label{figure:Conley} 
\end{figure}

  \begin{footnotesize}

\noindent 
\footnotesize{This work  is supported 
by PRIN 2020 (2020XB3EFL001)‚ Hamiltonian and dispersive PDEs‚ and 
PRIN 2022 (2022HSSYPN)‚ TESEO Turbulent Effects vs Stability in Equations from Oceanography. A. Maspero is supported by the European Union ERC CONSOLIDATOR GRANT 2023 GUnDHam,
Project Number: 101124921}
\normalsize

\medskip

\noindent 
{\it Tommaso Barbieri}, SISSA, Via Bonomea 265, 34136, Trieste, Italy, \texttt{tbarbier@sissa.it}.
 \\[1mm]
{\it Massimiliano Berti}, SISSA, Via Bonomea 265, 34136, Trieste, Italy, \texttt{berti@sissa.it},  
\\[1mm]
{   \it Alberto Maspero}, SISSA, Via Bonomea 265, 34136, Trieste, Italy,  \texttt{amaspero@sissa.it}, 
 \\[1mm] 
{   \it Marco Mazzucchelli}, ENS de Lyon, 46 all\'ee d'Italie, 69364 Lyon, France, 
 \texttt{marco.mazzucchelli@ens-lyon.fr}.

  \end{footnotesize}
 
\end{document}